\documentclass{amsart}
\usepackage{amsfonts,amscd,amssymb,latexsym}

\newtheorem{theorem}{Theorem}[section]
 \newtheorem{corollary}[theorem]{Corollary}
\newtheorem{lemma}[theorem]{Lemma}
 \newtheorem{proposition}[theorem]{Proposition}
 \theoremstyle{definition}
 
 \theoremstyle{remark}
 \newtheorem{remark}[theorem]{Remark}
 \newtheorem*{example}{Example}
 \numberwithin{equation}{section}

 \makeatother

\DeclareMathOperator{\Cdb}{{\mathbb C}}
\DeclareMathOperator{\Rdb}{{\mathbb R}}
\DeclareMathOperator{\Ddb}{{\mathbb D}}
\DeclareMathOperator{\Tdb}{{\mathbb T}}
\DeclareMathOperator{\Ndb}{{\mathbb N}}

\DeclareMathOperator{\A}{{\mathcal A}}
\DeclareMathOperator{\B}{{\mathcal B}}

\begin{document}

\title[Real positive maps on operator algebras]{Real positive maps and conditional expectations on operator algebras}

\author{David P. Blecher}
\address{Department of Mathematics \\ University of Houston \\ Houston, TX
77204-3008} \email[David P. Blecher]{dblecher@math.uh.edu}


\date{Revision of September 1, 2020}
\thanks{*Blecher was partially supported by a Simons Foundation Collaborative grant
527078.}

\dedicatory{Dedicated to the memory of  E. G. Effros and 
 Coenraad Labuschagne} 
\begin{abstract} Most  of this article is an  expanded version of our talk at the Positivity X conference.   It is essentially a survey, but some part, like most of 
the lengthy Section 
\ref{cp}, is comprised of new results whose proofs are unpublished elsewhere.
We begin by reviewing the theory of real positivity of operator algebras
initiated by the author  
and Charles Read.  Then we present several
 new general results (mostly joint work with Matthew Neal)
about real positive maps.
The key point is that real positivity is often the right replacement in a general algebra $A$ for positivity
in $C^*$-algebras.  We then apply this to studying contractive projections (`conditional expectations') 
and isometries of operator algebras.
  For example we  generalize and find variants of certain classical results on positive projections
 on $C^*$-algebras and JB algebras due
to Choi, Effros, St{\o}rmer, Friedman and Russo, and others. In previous work with Neal we had done
the `completely contractive' case; we focus here on describing the 
real positive contractive case  from recent work with Neal.    
We also prove here several new and complementary results on this topic due to the author, indeed this new 
work constitutes most of Section \ref{cp}.     Finally, in the last section we describe 
a related part of some recent  joint work with Labuschagne on what we consider to be a good noncommutative
generalization of the `characters' (i.e. homomorphisms into 
the scalars) on an algebra.   Such characters are a special case 
of the projections mentioned above, and are shown to be intimately related to conditional 
expectations.   The idea is to try to use these
to  generalize certain classical function algebra results 
involving characters.  
\end{abstract}

\maketitle


\section{Introduction: Real positivity}

Positivity plays a key role in physics, in fact one could say 
it is intrinsic to the structure of the (quantum) universe.  
 For lack of a 
better name, we shall use the term {\em quantum positivity}
to refer to the positivity found in the `standard model' of quantum mechanics,
that is, positivity for operators on 
a 
complex Hilbert space $H$, or positivity in algebraic systems comprised of Hilbert
space operators. It is indeed absolutely
fundamental and pervasive in quantum physics, modern analysis, noncommutative
geometry, and related fields.
The associated order on selfadjoint operators is sometimes
called the {\em L\"owner order}: $$S \leq T \; \; \; \; \; \; \; \;  \textrm{if and only if}  \; \; \; \; \; \; \; \; \langle S \zeta ,  \zeta  \rangle \leq \langle T \zeta ,  \zeta  \rangle  \;   \textrm{for all} \;   \zeta \in H.$$
There are several other well known characterizations of the positive cone 
in this order, i.e.\ for $T \geq 0$, for example in terms of the 
spectrum, or the numerical range, or in terms of a metric inequality, 
or an algebraic identity ($T = S^* S$, or $T = R^2$ where $R = R^*$), etc.  Here the $*$ is the usual adjoint on $B(H)$. 

The latter characterizations all make sense in any {\em $C^*$-algebra},
that is, a selfadjoint  
(that is, closed under the adjoint operation $*$)
norm closed subalgebra of the bounded
linear operators on $H$, or, more
 abstractly, a Banach $*$-algebra $A$ satisfying the $C^*$-identity 
$\| x^* x \| = \| x \|^2$.  We recall that a {\em von Neumann algebra} is a weak* closed $C^*$-algebra, or  abstractly a $C^*$-algebra with a Banach space predual.
These two classes of algebras are typically regarded as, respectively,
{\em noncommutative topology} and {\em noncommutative measure theory}.
Simplistically one could say that these are the kind of  noncommutative topology and measure theory needed for quantum physics.
The positivity and order above are one of the main ingredients of the vast theory of  $C^*$-algebras and von Neumann algebras.
There is a sense in which,
explicitly or implicitly, `quantum positivity'  underlies almost every proof in
$C^*$-algebra theory (see for example the texts
\cite{Blbook,P,Tak}).  

In an ongoing 
program (see e.g.\ \cite{BRI, BRII, BRord, BBS,BNp,Bsan,BOZ,Bsan,BPhilI,BWj,BNj,BT}), we have been importing some of this 
vast panorama of $C^*$-algebraic  positivity for use in more general algebras (Banach algebras, nonselfadjoint
operator algebras, Jordan operator algebras, etc).  The usual theory of `quantum positivity'  in operator algebras  is 
so spectacular and powerful it makes sense to 
make many of these tools available elsewhere.
To do this we use {\em real positivity} systematically.   
The main goal of the present paper is to describe some of the very recent updates in this program.
We do not prove many of the results stated here, just those that are not proved elsewhere. 
For example, most of Section \ref{cp} is proved here for the first time. 

In the present article, $H, K$ will denote Hilbert spaces over the
complex field. We write $B(H)$ for the algebra of bounded linear
operators $T : H \to H$. This is just the $n \times n$ matrix
algebra $M_n$ if $H$ is finite dimensional.
An operator $T \in B(H)$ is  {\em  real positive} (or  {\em  accretive}) if  Re $T = (T + T^*)/2 \geq 0$.
Again there are several other equivalent characterizations of real positive operators,
for example that the numerical range lies in the closed right half plane,
or they may be characterized by a metric inequality, 
or an algebraic identity, etc.  See e.g.\  \cite[Lemma 2.4]{Bsan}.
  The latter characterizations all make sense in any unital Banach algebra $A$ (by unital we mean that it 
  has an identity element of norm 1).
We write ${\mathfrak r}_A$ for the `cone' of real positive $T \in A$.
In \cite{BOZ,Bsan,BPhilI, BPhilII} we study 
real positivity in Banach algebras and $L^p$-operator algebras.  In \cite{BT} (written after the present survey was submitted)
we consider real positivity in real operator algebras, and real positive real linear maps.
We could have also reported here on the last three cited papers, in the spirit of recent updates and work in progress.  
However for the sake of not becoming too dispersed, in
 the present paper all of our algebras will  be {\em operator algebras}
or {\em Jordan operator algebras}.  

For us an operator algebra is  a norm-closed associative (but  
not necessarily selfadjoint) subalgebra of $B(H)$.  These were characterized abstractly in \cite{BRS}, and much of 
their general theory may be found in \cite{BLM}.
A   Jordan operator algebra is a  norm-closed linear subspace $A \subset B(H)$ 
which is closed under the
`Jordan product' $a \circ b = \frac{1}{2}(ab+ba)$ (or equivalently, with  $a^2 \in A$ for all $a \in A$).  The theory of Jordan operator algebras in this
sense is quite recent, and may be found in \cite{BWj,BNj,BWj2,ZWthes,BNjp}.  
There is a much older theory of {\em Jordan $C^*$-algebras} (also called $JC^*$-{\em algebras}).
These are the Jordan operator algebras  $A \subset B(H)$ which are also closed under the involution of $B(H)$.  
Indeed $JC^*$-algebras are historically essentially amongst the first examples of `operator algebras'. 
An  Annals paper of Jordan, von Neumann and Wigner from 1934 on these nonassociative algebras 
 \cite{JVW} begins with the line
 ``One of us has shown that the statistical properties of the measurements of a quantum mechanical system assume their simplest form when expressed in terms of a certain hypercomplex algebra which is commutative but not associative".
 Their hope was that such algebras ``would form a suitable starting point for a generalization of the present quantum mechanical theory".
 These days $JC^*$-algebras are often viewed within the larger theory of $JB^*$-algebras \cite{HS,Rod,Rod2}.   We will view them 
 within the class of Jordan operator algebras defined above. 

Of course every operator algebra is a  Jordan operator algebra.
The latter algebras turn out to be the correct most general setting for many of the results below.
Thus we state such results for Jordan operator algebras; 
the reader who does not care about nonassociative algebras should simply
restrict to the associative case.    Indeed the statement of many of our results 
contain the phrase `(Jordan) operator algebra', this invites the reader to simply ignore the word `Jordan'. 
 However a few of our new results do apply only to (associative) operator algebras.  

  The main principle for us is that real positivity is often  the right replacement in a general algebra $A$ for positivity in $C^*$-algebras.   To be honest, 
  there are many $C^*$-subtheories or results 
in  which this approach does not work.  We focus here on some of  the many settings where it does give
something interesting and behaves well.       Another main subtheme is that of `conditional expectation', as we shall see below together
with the relations between this theme and real positivity.

Turning to the structure of our paper, in Section 2 we begin by recalling 
the basics of real positivity  from our work with Charles Read (actually most results are stated in the 
later and more general setting from \cite{BWj}).  In Section 3 we survey some new and foundational 
results from \cite{BNjp} concerning  real positive maps, generalizing
some  aspects of the basic theory of positive maps on $C^*$-algebras \cite{ST}.
Considering such maps, as opposed to the `completely positive' or 
`completely contractive' case (terms defined below), forces
one into the more general  setting of Jordan operator algebras.
For example, the range of a positive projection (i.e.\ idempotent linear transformation) 
on a $C^*$-algebra need not be again be isomorphic to a $C^*$-algebra
(consider $\frac{1}{2}(x + x^T)$ on $M_2$), but it is always a Jordan operator algebra.     Also,  as one sees already in
Kadison's Banach-Stone  theorem for
$C^*$-algebras \cite{Kad}, isometries of $C^*$-algebras relate to Jordan $*$-homomorphisms and not necessarily to
$*$-homomorphisms.  We recall that a 
Jordan algebra homomorphism is a map satisfying $T(a \circ b) = T(a) \circ T(b)$
(or equivalently, with  $T(a^2) = T(a)^2$) for all $a, b \in A$.  

Indeed in Section \ref{qbs} 
 we also describe a new 
Banach-Stone type theorem from \cite{BNjp} for nonunital isometries between 
  Jordan operator algebras;  characterizing 
 such isometries in the spirit of Kadison's Banach-Stone  theorem mentioned above. 
This result is 
needed in 
the proof of a theorem stated towards the end of Section \ref{cp}, the characterization of symmetric real positive projections.
 It  requires 
an analysis of `quasi-multipliers' of (Jordan) operator algebras, a nontrivial link between the latter and quasi-multipliers of any 
generated $C^*$-algebra, 
a little known $C^*$-algebra theorem about quasi-multipliers due to Akemann and Pedersen,
and some theory of Jordan multiplier algebras.   In Section \ref{cp} we study real positive conditional expectations, and contractive (i.e.\ norm $\leq 1$)
and bicontractive 
 projections, on operator  algebras or 
Jordan operator  algebras.    
In earlier work \cite{BNp} we had considered completely contractive and completely bicontractive projections; the focus in Section \ref{cp}
 is how much of this is still true with the word `completely' removed.    Some of the main questions here 
concern generalization of famous results of Tomiyama and Choi and Effros on a $C^*$-algebra $A$:
The range of a positive contractive projection  $P$ from $A$ onto a $C^*$-subalgebra is a {\em conditional expectation},
by which we mean that $P(a P(b)) = P(a) P(b) = P(P(a) b)$ for all $a, b \in A$.    
(We discuss why these are called conditional expectations in Section \ref{wace}, and also review 
some aspects of the history and theory of classical (probabilistic) conditional expectations there.)
And even if $P(A)$ is not a subalgebra, it is still
a  $C^*$-algebra in the canonical  product $P(ab)$ defined by $P$ (assuming $P$ is completely positive, otherwise 
we are in the Jordan situation mentioned in the last paragraph), and with respect to 
this product $P$ is still a `conditional expectation'.   This is the Choi and Effros result,
and the latter product is often called the Choi-Effros product or $P$-product.  Pioneering results about contractive
projections on JB*-algebras may be found in \cite[Theorem
2.21]{R83}  and \cite[Corollary 1]{K84}.  In Section \ref{cp} we consider similar questions for
 a real positive contractive projection $P$ on an operator algebra or  Jordan operator algebra: is it a conditional expectation if the range is a
 subalgebra or Jordan subalgebra, in the general case 
 is $P(A)$  again an operator algebra or  Jordan operator algebra in the canonical  product defined by $P$,
 and is $P$ a `conditional expectation'?   
 Some of Section \ref{cp} 
surveys a selection of some results from Sections 4--6 of \cite{BNjp}.
However, as we said earlier, most  of Section \ref{cp}  of the present paper 
consists of new results and proofs that do not appear elsewhere.    

 In Section \ref{wace} we discuss
the concept and history of `conditional expectation', and review very briefly 
some features of the important but difficult 
theory of normal conditional expectations of von Neumann algebras.
In the last section we describe 
some joint work with Labuschagne on what we consider to be a good noncommutative
generalization of the `characters' (i.e.\ homomorphisms into 
the scalars) on an algebra.   Such characters are a special case 
of the projections mentioned above.   The idea is to try to use these
to  generalize certain classical function algebra results 
involving characters.   We focus here on some aspects of this work that relate to material and results
from earlier in the present article.   For example, we generalize the von Neumann algebraic setting of conditional expectations, which involves 
an inclusion $D \subset M$ of von Neumann algebras, to inclusions $D \subset A \subset M$ for a
weak* closed subalgebra $A$.   This is a setting where one can find a positive  (in the usual sense) 
extension of a real positive map which does not increase  the range of the map.
(The example to bear in mind is from Hardy space theory:
$\Cdb 1 \subset H^\infty(\Ddb) \subset L^\infty(\Tdb)$.)

\section{Some general results on real positivity}  \label{sgr} 
For operator algebras or Jordan operator algebras the definition of `real positive' (i.e.\  $T + T^* \geq 0$) does not depend of the particular representation of the algebras.
Indeed as we said earlier there are nice equivalent definitions of `accretive' which
 make this point clear.
For example, an element $x$  is real positive if 
and only if $\Vert 1 - tx \Vert \leq 1 + t^2 \Vert x \Vert^2$  for all $t > 0$.
Here, the $1$ above is well-defined,  for every nonunital nonselfadjoint
(even Jordan) operator algebra $A$.  This is essentially `Meyers theorem': in \cite[Section 2.2]{BWj}
we established using a result of Meyer 
 that every Jordan operator algebra $A$ has a  unitization $A^1$
which is unique up
to isometric Jordan homomorphism.

It is useful that the `real positive cone' ${\mathfrak r}_A$ above has a stronger subcone $\Rdb_+ {\mathfrak F}_A$, where
 ${\mathfrak F}_A = \{ a \in A : \Vert 1 - a \Vert \leq 1 \}$.
This is a proper cone (in the sense that $0$ is the only element $x \in A$ with
both $x$ and $-x$ in this cone). 
Again, the $1$ here is well-defined,  as we explained above.

In \cite{BRI,BRII,BRord} Charles Read and the author began using and 
systematically developing  real positivity in operator algebras, proving results
like the following.  We will however state the more general analogous results for Jordan operator algebras
from \cite{BWj}.   Thus throughout the rest of this section $A$ is a Jordan operator algebra.

\begin{proposition} \label{2co} ${\mathfrak r}_A = \overline{ \Rdb_+ {\mathfrak F}_A}$.  
\end{proposition}  

Sometimes  one first has to prove a positivity result using ${\mathfrak F}_A$, and then
use the last Proposition to generalize to all real positive elements.   

We write {\em cai} for short for an  approximate identity which is  contractive (that is, of norm $\leq 1$).
A {\em Jordan contractive approximate identity}  for a Jordan 
operator algebra   is a net
$(e_t)$  of contractions with $e_t \circ a \to a$ for all $a \in A$.    If $A$ is an associative operator algebra
then one can show that the existence of a Jordan cai implies existence of a (two-sided) cai.

\begin{theorem} \label{cairp}   $A$ has a  (Jordan) cai if and only if $A$ has a real positive (Jordan) cai, and 
 if and only if the real positive elements in $A$ span $A$.   
\end{theorem}

In the Jordan algebra case this follows from \cite[Theorem 4.1]{BWj}.
The (earlier) operator algebra case is in \cite{BRI}; this result suggests that the operator algebras with a cai are for some purposes 
the  good generalization of $C^*$-algebras (recall that 
all $C^*$-algebras have a cai).
We call a (Jordan) operator algebra  with a cai  {\em approximately unital}.
We focus on this class here, although there are many things one can do for 
algebras without a cai. 

We recall that an operator $T : A \to B$ between $C^*$-algebras (or {\em operator systems},
that is, selfadjoint unital subspaces
of $C^*$-algebras) is {\em completely
positive} if $T(A_+) \subset B_+$, and similarly at the matrix levels.  That is $T$ acting entrywise takes $M_n(A)_+$ into $M_n(B)_+$ for all
$n \in \Ndb$.  Recall that $M_n(A)$ is also a $C^*$-algebra if $A$ is a $C^*$-algebra.  
Similarly $T$ is {\em completely contractive} if this entrywise action on $M_n(A)$ is contractive for all $n$. 
A linear map $T : A \to B$ between (Jordan) operator
algebras or unital operator spaces is {\em real completely
positive}
 if $T({\mathfrak r}_{A}) \subset {\mathfrak r}_{B}$
and similarly at the matrix levels
(i.e.\ ${\mathfrak r}_{M_n(A)}$ is taken to 
${\mathfrak r}_{M_n(B)}$ for all
$n \in \Ndb$).

 (Trace preserving) completely positive maps are the standard way to take a  measurement
in quantum information theory.   They appear explicitly or implicitly in 
the important definitions of {\em   quantum channels,} or   {\em POVM's},  or 
{\em quantum instruments} in that subject.  
Completely positive maps are also often viewed as the correct `quantum analogue' of 
positive maps between function spaces.  Much of the present paper
however focuses on positivity as opposed to complete positivity.
  So mostly we will consider 
 maps/isometries/projections
that are simply {\em real positive}, i.e.\ ${\mathfrak r}_{A}$ 
is taken to ${\mathfrak r}_{B}$.
This will be less interesting for some (certainly for the author,
who usually works on things related to operator spaces,
 see e.g.\ \cite{BLM}).   We justify this for the present article  for two reasons.   First, 
positive maps are quite interesting, for example in the Jordan approach to 
quantum physics or entanglement, etc, 
even if they are less important in many applications.
Though they may be poorer relatives, they are often fine fellows, and can go places that 
the other cannot.  Moreover even in the theory of complete positivity,
maps that are positive but not completely positivity often make an appearance.     In the physics literature this
happens for example in the study of entanglement
where  positive but not completely positivity  maps  are needed (i.e.\ necessary and sufficient) to test if a fixed given mixed state is entangled or separable 
\cite{HHH}.   There 
are papers with titles like ``Who’s afraid of not completely positive maps?" (by Sudarshan et al). Positive but not completely positive maps are sometimes important in the theory.  
Thus the theory of general positive maps deserves
to be worked out, and that is the setting of much of the present article.
The second justification for looking at positivity as opposed to complete positivity
here is that a main goal of the present paper is to describe the most recent updates,
and recently we have been investigating  what transpires if the word
`completely' is dropped.

\begin{theorem} \label{stinei1}  \cite{BRI,BBS}   
A (not necessarily unital) linear map $T : A \to B$ between $C^*$-algebras or operator
systems is completely positive in the usual sense if and only if it is real completely
positive.   Also, $T$ is real positive if and only if it is positive in the usual sense.  
\end{theorem}

Thus one does not lose anything from the (completely) positive theory when considering real (completely)
 positive maps.

\begin{theorem} \label{stine2}  {\rm (Extension and Stinespring-type result), \cite{BRI,BBS,BWj}} \ A linear map $T : A \to B(H)$ on
an  approximately unital Jordan operator
algebra or unital operator space is  real completely
positive if and only if  $T$ has a completely positive (in the usual sense) extension $\tilde{T} : C^*(A) \to B(H)$.  Here $C^*(A)$ is any $C^*$-algebra generated by $A$.
This is equivalent to being able
to write $T$ as the restriction to $A$ of $V^* \pi(\cdot) V$ for a
$*$-representation $\pi : C^*(A) \to B(K)$, and an operator $V : H \to K$. 
  \end{theorem}

This last result generalizes to a larger class of algebras both   the original  Stinespring theorem \cite{Sti}, and the 
Arveson extension theorem for completely positive maps \cite{Arv,Pau}.    Of course the range of $\tilde{T}$ will typically be much larger than the range 
of $T$.   In the final pages of this article we will describe a setting
where the extension $\tilde{T}$ surprisingly has the same range as $T$.  

Since it is easy to see that every real positive scalar valued functional 
on such a space $A$ is real completely
positive, it follows that such real positive functionals
are simply the restriction to $A$ of positive scalar multiples
of states (that is, positive norm 1 linear functionals)  on $C^*(A)$.   Thus we obtain a generalization 
of the famous GNS theorem: A functional $\varphi : A \to \Cdb$ is real positive if and only if 
there is a $*$-representation $\pi : C^*(A) \to B(K)$, and a vector $\zeta \in H$ with 
$\varphi(x) =  \langle \, \pi(x) \, \zeta ,  \zeta  \rangle$  for $x \in A$.

The ordering induced on $A$ by the real positive cone is obviously
$b \preccurlyeq a$ iff $a-b$ is real positive.  One may then go ahead and
follow the $C^*$-theory by studying for example the order theory in the unit ball.
Order theory in the unit ball of a $C^*$-algebra, or of its dual, is crucial  in $C^*$-algebra theory.
A feature of the first such result below is that having the order theory is  possible if
and only if there is a contractive approximate identity around.
On the other hand the following couple of results are fairly obvious
for unital algebras.  

\begin{theorem} \label{orderth}  Let  $A$ be an  (Jordan)  operator algebra which generates a $C^*$-algebra
$B$, and let ${\mathcal U}_A = \{ a \in A : \Vert a \Vert < 1 \}$.  The following are equivalent:
\begin{itemize} \item [(1)]   $A$ is approximately unital.
 \item [(2)]  For any {\em positive} $b \in {\mathcal U}_B$ there exists  a real positive $a$
with $b \preccurlyeq  a  \preccurlyeq 1$.
\item [(3)]   For any pair
$x, y \in {\mathcal U}_A$ there exists a   real positive  contraction
$a$ 
with $x \preccurlyeq a$ and $y \preccurlyeq a$.
\item [(4)]    For any $b \in {\mathcal U}_A$  there exists a  real positive contraction
$a$ 
with $-a \preccurlyeq b \preccurlyeq a$.
\item [(5)] For any $b \in {\mathcal U}_A$  there exists a real positive contractions
$x, y$ with $b = x-y$.
\item [(6)]  ${\mathfrak r}_A$ is a generating cone (that is, $A = {\mathfrak r}_A - {\mathfrak r}_A$).
\end{itemize}
\end{theorem}

The real positive elements above may be chosen `nearly positive' and in $\frac{1}{2}  {\mathfrak F}_A$.  

\begin{corollary}  \label{cofi}   Thus if an approximately unital  (Jordan) 
 operator algebra $A$ generates a $C^*$-algebra
$B$, then $A$ is {\em order cofinal} in $B$: given $b \in B_+$ there exists $a \in A$ with $b \preccurlyeq a$.   Indeed we can do this with
$b \preccurlyeq a \preccurlyeq \Vert b \Vert + \epsilon$.
\end{corollary} 

We will use this result later.

We recall that the positive part of the  open unit ball ${\mathcal U}_B$ of a $C^*$-algebra $B$
is a directed set, and indeed is a  net which is a positive cai for $B$.  The following generalizes this to operator algebras:

\begin{corollary}  \label{pcai}  If $A$ is an approximately unital (Jordan) operator algebra, then the real positive strict contractions, indeed
the set of real positive elements $\{ a \in A : \| a \| < 1 , \| 1 - 2a \| \leq 1 \}$, is a directed set in the $\preccurlyeq$ ordering, and with this ordering this set  is an  increasing
cai for $A$.  
\end{corollary}  

It is also interesting to consider order theory in the dual of an (Jordan) operator algebra but we will not do so here (see e.g.\ \cite{BRord}).
In most of the remainder of the paper we turn to  real positive maps and projections on operator algebras.

  \section{Real positive maps} \label{Sec3} 

We recall again that a linear map $T : A \to B$ is  real positive if $T({\mathfrak r}_A)
\subset {\mathfrak r}_B$.  We saw in  Theorem \ref{stinei1}
that the real positive maps 
on $C^*$-algebras or operator systems are just  the positive maps
in the usual sense.
The following recent results are a sample from Section 2 in \cite{BNjp}.  
 We have chosen to include several of these results partly because they are used in later
theorems, and it is helpful to see how the theory builds on itself.   We also remark that there are (historically earlier) 
operator space versions of most of the following results (with the word `completely' added in many places), but that 
is not our focus here.  The interested reader
can find these in our papers referenced in the introduction, e.g. \cite{BNp}.

\begin{proposition} \label{normpo} 
If $T : A \to B$ is a  real positive linear map  between unital  (resp.\ approximately unital)  Jordan operator algebras
then $T$ is bounded and
$\| T \| = \| T(1) \|$  (resp.\ $\| T \| = \sup_t \, \| T(e_t) \| = \| T^{**}(1) \|$),
if $(e_t)$ is a (Jordan) cai for $A$.
\end{proposition}  

For us a contraction or contractive map means it has norm $\leq 1$.

\begin{proposition}   \label{onladdn} Contractive homomorphisms (resp.\ Jordan homomorphisms)  
between approximately unital operator algebras (resp.\ Jordan operator algebras)
are  real positive.

A unital  contraction between unital (Jordan)
operator algebras is real positive.
\end{proposition}

\begin{theorem} \label{extn}   Let $A$ and $B$ be approximately unital (Jordan) 
operator algebras, and write
$A^1$ for a (unique, as we said at the start of 
Section {\rm \ref{sgr}}) unitization of $A$.
If $A$ is unital choose  $A^1 = A \oplus^\infty \Cdb$ as usual (we do not necessarily make this requirement on $B$).
A real positive contractive  linear map  $T : A \to B$ extends to a unital real
 positive contractive linear map from $A^1$ to $B^1$. 
\end{theorem} 
 
Thus real positive contractions from  $A$ to $B$ are just the restrictions of unital contractions  from $A^1$ to $B^1$.   This `is a theorem' because it seems quite nontrivial.
Indeed our proof uses the earlier  nontrivial `order 
theoretic' fact of cofinality in
Corollary \ref{cofi}. 
Moreover this theorem is foundational and seems to be
exceedingly useful, very often permitting 
a `reduction to the unital case'.   This trick may not have been noticed 
for example in the Jordan $C^*$-algebra literature, where sometimes complicated
new arguments are used to deal with approximately unital algebras rather than a 
quick appeal to the unital case.


\begin{remark}  The last theorem is not true in general
 if $A$  is not approximately unital.
For example, consider $A_0(\Ddb)$, the continuous functions on the disk that are analytic inside the disk
and which vanish at $0$. 
If Re $f \geq 0$ and $f(0) = 0$ then Re $f$ is identically $0$ by the maximum modulus theorem for harmonic
functions.  Hence $f$ is constant and zero.
Thus the map $f \mapsto f'(0)$ is  trivially real positive, and it is a contraction by the Schwarz inequality.  However, the unital extension of the latter map
is not a contraction, or equivalently is not positive.   Indeed the states on $A(\Ddb) = A_0(\Ddb)^1$ are integrals against probability measures $\mu$ on the circle.
Nonetheless, the existence of a
positive measure  $\mu$ on the circle with  $\int_{\Tdb} \, z^n \, d \mu = 0$ for $n \geq 2$, and $\int_{\Tdb} \, z \, d \mu = 1$,
is ruled out by the solution to the well known trigonometric moment problem, since e.g.\ the $3 \times 3$ Toeplitz matrix whose entries are all  $1$ except
for zeroes in the $1$-$3$ and $3$-$1$ entries is not positive.
Thus the theorem fails in this case.   
\end{remark} 

The following result, whose proof relies in an interesting way on 
Theorem \ref{extn}, is useful
for questions about real positivity because it  shows that we can
often get away with
 working with
the simpler set ${\mathfrak F}_A = \{ x \in A : \| 1 - x \| \leq 1 \}$.

\begin{corollary}   A linear map $T : A \to B$ between approximately unital
(Jordan) operator algebras
is real positive and contractive if and only if $T({\mathfrak F}_A) \subset {\mathfrak F}_B$.
\end{corollary}

\begin{corollary} Let $M$ be a unital weak* closed operator
algebra, $\Phi : M \to M$ be a weak* continuous real positive contraction, and let $M^\Phi$ be the weak* closed unital 
subspace of  fixed points of $\Phi$.   Then there exists a  real positive contractive projection on $M$ with range  $M^\Phi$.
\end{corollary}

The space $M^\Phi$ is sometimes called a {\em Poisson boundary}.  We assign  $M^\Phi$  the new product $\Phi(xy)$ (or $\Phi(x \circ y)$ in the Jordan case).
It follows from the idea in the next proof and results such as Theorems \ref{cdsproj} and \ref{ipscor} 
and others in Section \ref{cp} below that  $M^\Phi$  
is typically a Jordan operator algebra.    Conversely, it may be an interesting question as to whether all
Jordan operator algebras arise in this way as fixed points of real positive contractions.  

\begin{corollary}   \label{Poisson} Let  $M$ be a unital weak* closed operator
algebra (resp.\ Jordan operator algebra), and let 
$\Phi : M \to M$ be a weak* continuous real completely positive complete contraction.  Then  there exists a real completely positive 
completely contractive projection on $M$ with range  $M^\Phi$, and $M^\Phi$ with its new product is a unital dual operator algebra in the sense of \cite[Section 2.7]{BLM} (resp.\ is 
a unital Jordan operator
algebra). \end{corollary}

\begin{proof}   
One may follow the proof in \cite[Corollary 1.6]{ES}, taking weak* limits in  the unit ball 
of $CB(M,M) = (M \hat{\otimes} M_*)^*$ of averages of powers of $\Phi$.  One obtains a completely contractive projection $P$ on $M$ with range  $M^\Phi$.
It is an exercise in weak* approximation that $P$ is real positive, and similarly it is real completely positive.
If $M$ is a unital operator
algebra then  $M^\Phi$ is an operator
algebra in the $P$-product by \cite[Theorem 2.5]{BNp}, with identity $P(1)$.    
Since $M^\Phi$ consists of the fixed points of $\Phi$ it is weak* closed.
Hence it is a dual operator space, thus  a  dual operator algebra in the sense of \cite[Section 2.7]{BLM}, by Theorem 2.7.9 in that reference.
The Jordan case follows similarly from 
\cite[Theorem 4.18]{BNjp}.    \end{proof}

A similar result holds for weak* continuous complete contractions using Theorem \ref{ccproj} (1) in the proof in place of \cite[Theorem 2.5]{BNp}.

The following is a nonselfadjoint analogue of the well known fact that
the positive part of the kernel of a
 positive map $T$ on a $C^*$-algebra $B$ has the following `ideal-like' property:
$$T(xy) =  T(yx) = 0 , \qquad y \in {\rm Ker}(T)_+ , \; \, x \in B.$$
  Note that the entire kernel is rarely an ideal.

  \begin{lemma} Suppose that  $A$ is an approximately unital  operator algebra
and that $T : A \to B(H)$ is a real positive map on $A$.  If $x \in A$ and $y \in {\mathfrak r}_{A} \cap {\rm Ker}(T)$  then $xy$ and $yx$ are in ${\rm Ker}(T)$. \end{lemma} 

The following, which is needed later e.g.\ for Theorems \ref{cdsproj} and \ref{bicr}, is a sample corollary of this:

\begin{corollary}   If $A$ is an approximately unital Jordan operator algebra,  $T : A \to B(H)$ is real positive, and if $J = {\rm Ker}(T)$  is contained in the 
closed Jordan
 subalgebra generated by the real positive elements that $J$ contains,
  then ${\rm Ker}(T)$
is  an approximately unital Jordan ideal in $A$. \end{corollary}

In fact one may replace `Jordan
 subalgebra' in the last result with `Jordan hereditary 
 subalgebra'.   A Jordan hereditary 
 subalgebra of a Jordan
operator algebra $A$ is a closed approximately unital Jordan subalgebra $D$ satisfying $dAd \subset D$ for all $d \in D$. 

The following result, whose proof benefits from some ideas from  Lemmas 2.8 and 4.6 in  \cite{BNjp}, will be also used later in Section \ref{cp}.   

\begin{lemma} \label{ishe3} Let $A$ be  a unital 
operator space (resp.\ approximately  unital Jordan
operator algebra),
and let  $T : A \to B(H)$ be a unital  (resp.\ real positive) contraction.
Suppose that $e$ is a projection in $A$ with $e \circ A \subset A$,
such that $q = T(e)$ is a projection in $B(H)$.   Then
$T(eae) = q T(a) q$ 
and  $T(a \circ e) =T(a) \circ q$ for all $a \in A$.     \end{lemma}
 
\begin{proof}     
By Theorem  \ref{extn} we can unitize if necessary.
So we may assume that $T(1) = 1$ (since $e \circ A^1 \subset A^1$).
Let $S = q T(\cdot) q$.
Then $S$ is real positive and $S(1) = q = S(e)$.   By Lemma
 2.8 in \cite{BNjp}, $S(a) = q T(a) q = q T(eae) q$. 

For any $a \in {\rm Ball}(A)$
we have $\| 1 - e \pm eae \| \leq 1$, so that 
$\| 1- q \pm T(eae)  \| \leq 1$.  Hence $\| q^\perp
\pm  T(eae)  \| \leq 1$.
Since $q^\perp$ is an extreme point of $q^\perp B(H) q^\perp$
we see that $q^\perp T(eae) q^\perp = 0$.
Looking at the matrix of $T(eae)$ with respect to $q^\perp$,
and using  $\| q^\perp
\pm q^\perp T(eae)  q^\perp \| \leq 1$, we also see that 
 $T(eae) = q T(eae) q$. So we have
proved that $T(eae) = q T(a) q$.    Similarly, $T(e^\perp ae^\perp) = q^\perp T(a) q^\perp$. 

The second identity follows from  the facts in the last line, and from the identity $a \circ q  =  
\frac{1}{2} (a + qaq - q^\perp a q^\perp ),$ and the similar identity with 
$q$ replaced by $e$.   
  \end{proof}
  
The last few results are a sample of recent  results on  real positive maps (mostly 
from \cite{BNjp}).  Almost all of these particular results  are ingredients
of proofs of results featured in the rest of the paper. 

\section{Quasimultipliers and a  Banach-Stone theorem} \label{qbs} 

There are very many
Banach-Stone type  theorems in the literature. 
For example, we already mentioned 
Kadison's result  
that surjective linear isometries between $C^*$-algebras are precisely the maps
$u \pi(\cdot)$ for a surjective Jordan $*$-isomorphism $\pi$  and unitary multiplier $u$ \cite{Kad}.
In particular linearly isometrically isomorphic  $C^*$-algebras are Jordan $*$-isomorphic.
By spectral theory (by a result of Harris \cite[Proposition 3.4]{Har2} if necessary,
or see e.g.\ Proposition 3.4.4 in \cite{Rod})
 one can see that  the converse is true, Jordan $*$-isomorphic $JC^*$-algebras
are isometrically isomorphic.

 \begin{theorem}  \label{AS} {\rm (Arazy-Solel) \cite{AS}} \ Surjective unital linear isometries between unital (Jordan)   operator algebras
are  Jordan homomorphisms. \end{theorem}

A variant of this due to Arazy where the isometry need not take $1$ to $1$ may be found in \cite[Theorem 3.1]{Av}.

\begin{theorem} \cite{BWj}  An isometric surjection $T$ between approximately unital (Jordan)  operator algebras is real
positive if and only if $T$ is a Jordan algebra homomorphism.
 \end{theorem}

We now  wish to extend Arazy and Solel's result above to the case of  nonunital
 surjective isometries $T : A \to B$ between approximately unital
(Jordan) operator algebras.   This is needed in a theorem towards the end of the section, the characterization of symmetric real positive projections.
It turns out that the `obvious thing' is wrong in the Jordan   operator algebra case.
Although `unitally linearly isometrically isomorphic'  unital $JC^*$-algebras
are Jordan $*$-isomorphic (this follows e.g.\ from the next theorem, although it is much older of course, due 
again to Harris \cite{Harris0}), 
there exist linearly isometric unital $JC^*$-algebras which are not Jordan isomorphic
(see e.g.\ \cite[Antitheorem 3.4.34 and 
Corollary 3.4.76]{Rod}).

\begin{proposition} \label{noci} There exist linearly completely isometric  unital $JC^*$-algebras which are not Jordan isomorphic.
There exist unital $JC^*$-algebras which are Jordan $*$-isomorphic but not linearly completely isometric.
\end{proposition} 

\begin{proof}   The second statement has a fairly obvious putative counterexample: take the map $d(x) =  (x,x^{\intercal})$ from $M_n$ to $M_n \oplus M_n$.  
The range is a unital $JC^*$-algebra and $d$ is a Jordan $*$-isomorphism.   Here is one way to see that the range $R$  is not 
completely isometric to $M_n$.
By way of contradiction, suppose that $u : M_n \to R$ is a linear complete isometry.
Then it is easy to see that $$\| [ x_{ij}] \| =  \| [ u(x_{ij})] \| = \| [ u(x_{ij})]^\intercal \|  = \| [ u(x_{ji})] \| = \| [ x_{ji}] \| = \| [x_{ij}^\intercal ] \|, $$
for $[ x_{ij}] \in M_m(M_n)) .$
In other words, `transpose' is a complete isometry on $M_n$.   But this is well known to be false (take $x_{ij}$ above to be the usual matrix unit basis 
of $M_n$). 

For the first statement  we just need to show that the two algebras in \cite[Antitheorem 3.4.34]{Rod} are actually completely isometric.
These two algebras are two copies of a single Jordan $*$-subalgebra $A$ of a $C^*$-algebra $C$, but the second copy $B$ is given the 
Jordan product $\frac{1}{2}(x u^* y + y u^* x)$, and involution $u x^* u$, for a unitary $u$ in $C$ with $u$ and $u^*$ in $A$.   Right multiplication by $u^*$ on $C$ 
restricts to a completely isometric Jordan $*$-isomorphism (since $(xu^*)^* = (u x^* u) u^*$) from $B$ onto the Jordan $*$-subalgebra $Au^*$ of $C$.
(See also \cite[Lemma 3.1]{BNjp}.)   Thus $A$ and the latter Jordan $*$-subalgebra are 
completely isometric, but not Jordan isomorphic (since $A$ and $B$ are not Jordan isomorphic).
\end{proof}

Therefore
Banach-Stone theorems for nonunital isometries between Jordan operator algebras are not going
to look quite as one might first expect: one cannot expect the Jordan isomorphism appearing in the
conclusion to map onto the second $C^*$-algebra exactly. Nonetheless we  obtain a reasonable
Banach-Stone type  theorem for nonunital isometries between Jordan operator algebras.   This Banach-Stone type  theorem  (from \cite{BNjp}) plays a crucial role
in one of the theorems in the next section (Theorem \ref{symm}).
One of the main steps is to show that for $T$ as above,   $T^{**}(1)$ is in the Jordan multiplier algebra
$JM(B) = \{ x \in B^{**} : x \circ B \subset B \}$.  But this is not at all clear, and requires 
an analysis of `quasi-multipliers' of (Jordan) operator algebras, a nontrivial link between the latter and quasi-multipliers of any 
generated $C^*$-algebra, 
a little known $C^*$-algebra theorem about quasi-multipliers due to Akemann and Pedersen,
and some theory of Jordan multiplier algebras.
We define a {\em  quasimultiplier} of  $B$ to be
an element $w \in B^{**}$ with $bw b \in B$ for all $b \in B$.   Some of the steps in the proof of 
the next theorem are: showing that $T^{**}(1)$ gives rise to a quasimultiplier of $B$, and that 
such quasimultipliers are in   $JM(B)$ (which uses several of the ingredients mentioned a few 
lines back). 

Let $\Delta(A) = A \cap A^*$, sometimes called the `diagonal' of $A$.   One can show that this is a well defined  $C^*$-algebra (or $JC^*$-algebra) independently
of the particular representation of $A$ on a Hilbert space \cite{BWj}.   If $B$ is a $C^*$-algebra then $JM(B)$  is just the usual  $C^*$-algebraic multiplier algebra
$M(B)$.  This follows from e.g.\ \cite[Proposition 5.10.96]{Rod2}.   We recall that for an approximately unital operator algebra $D$, the multiplier
algebra  $M(D)$ may be defined to be the unital operator algebra $\{ x \in D^{**} : x D + Dx \subset D \}$.

We recall that a $JW^*$-algebra is a weak* closed $JC^*$-subalgebra of $B(H)$ (or of a von Neumann algebra).
The bidual of a  $JC^*$-algebra $A$ is a $JW^*$-algebra.   Indeed it is a weak* closed $JC^*$-subalgebra of $B^{**}$ if $A$ is a $JC^*$-subalgebra of a
$C^*$-algebra $B$.

\begin{theorem} \label{bs1}    Suppose that $T : A \to B$ is an isometric surjection between approximately unital Jordan operator algebras.
Suppose that $B$ is   a Jordan subalgebra of an (associative) operator algebra $D$,
and that $B$ generates $D$ as an operator algebra.
Then there exists a unitary $u \in \Delta(JM(B))$ which is also in
$\Delta(M(D))$,   and  there exists an isometric  Jordan  algebra homomorphism
$\pi : A \to D$, such that  $\pi(A) =  B u^*$ is a Jordan subalgebra of $D$,
and $$T = \pi(\cdot) u .$$ \end{theorem}

As is usual with noncommutative Banach-Stone theorems, we can also write the unitary $u$ on the left:
$T = u \, \theta(\cdot)$ (indeed simply set $\theta = u^* \pi(\cdot) u$).  

In  the unital case the following consequence for associative operator algebras  also follows from \cite[Theorem 3.1]{Av} (see also
\cite[Proposition 3.12]{CR18}):

\begin{corollary} \label{bs2}    Suppose that $T : A \to B$ is an isometric surjection between approximately unital operator algebras.
Then there exists a unitary $u \in \Delta(M(B))$ 
and  there exists an isometric  Jordan  algebra homomorphism
$\pi : A \to B$, such that $T = u \pi(\cdot) .$
\end{corollary}

Of course if $T$ is a complete isometry in the last results then so will be $\pi$.  

\section{ Contractive projections on operator algebras} \label{cp}  

The following theorem (essentially due to   Effros and St{\o}rmer \cite{ES}  in the case that $P(1) = 1$) 
shows what happens in the case of selfadjoint Jordan operator algebras ($JC^*$-algebras). 
The reader could take $A$ to be a $C^*$-algebra if they wish.  
The new case of the theorem, i.e.\ the case that $A$ is nonunital, or that $P(1) \neq 1$,
can be dealt with by passing to the unitization by using  Theorem \ref{extn}.  

\begin{theorem} \label{frs}   If $P : A \to A$ is a  positive contractive 
projection on a $JC^*$-algebra $A$ then $P(A)$ is a $JC^*$-algebra in the new product $P(x \circ y)$,
$P$ is still positive as a map into the latter $JC^*$-algebra,
and $$P(P(a) \circ P(b)) = P(a \circ P(b)) \; \; , \qquad a, b \in A \, .$$ 
If in addition $P(A)$ is a Jordan subalgebra of $A$ then $P$ is a Jordan conditional expectation: that is, 
$$P(a \circ P(b)) = P(a) \circ P(b)  \; \; , \qquad  a, b \in A.$$
\end{theorem}

We remark that the variant of the last 
theorem with JB*- instead of JC*- was later proved in \cite[Theorem
3.2]{CR20}.   The latter paper was in part inspired by our results in \cite{BNj}.

Our goal now is to try to generalize such results to more general algebras.
The following result is from \cite{BNj}: 

\begin{theorem} \label{ccproj} Let $A$ be a (Jordan)  operator algebra, and
$P : A \to A$ a completely contractive projection.
  \begin{itemize} \item  [(1)] The range of $P$ with product
$P(x \circ y)$, is  completely   isometrically Jordan isomorphic to a Jordan operator algebra.  \item  [(2)]  If $A$ is an associative 
operator algebra then the range of $P$ with product $P(x  y)$, is  completely isometrically algebra isomorphic to an
associative operator algebra.
\item  [(3)]  If $A$ is unital  
  and $P(1) = 1$ then the range of  $P$, with product
$P(x \circ y)$,
is unitally completely isometrically Jordan isomorphic to a unital  Jordan operator algebra.
  \end{itemize}
\end{theorem}

By the {\em $P$-product} or {\em new product} on $P(A)$
we mean the bilinear map $P(x \circ y)$ (in the Jordan algebra case) or
$P(x  y)$ (in the algebra case).

We do not recall having seen the following explicitly in the literature.  In any case, 
in view of the present venue, we give a simple proof of it using Theorem \ref{frs}.  

\begin{theorem}  \label{Sak}   If $A$ is
a $JC^*$-algebra which has a Banach space predual, then $A$ is
Jordan $*$-isomorphic, via a weak* homeomorphism, to a $JW^*$-algebra.
\end{theorem}

\begin{proof}  A $JC^*$-algebra with a predual has an identity (by e.g.\ the Krein-Milman theorem and 
\cite[Theorem 4.2.36]{Rod}).   Now $A^{**}$ is a $JW^*$-algebra as we showed above
Theorem \ref{bs1}. Suppose that $E^* = A$.  The canonical map $E \to E^{**} =
A^*$ dualizes to give a weak* continuous contractive unital, hence positive and $*$-linear, surjection
$\Phi : A^{**} \to A$. Regard $A$ as a $JC^*$-subalgebra of $A^{**}$.
It is easy to check that $\Phi$ extends the identity map on $A$, so
that $\Phi \circ \Phi = \Phi$. Thus $\Phi$ is a weak* continuous
`conditional expectation' satisfying Theorem \ref{frs}.  Applying that result we have
  $x \circ y$ and $y  \circ x$ are in Ker$(\Phi)$ for any $x \in
A, y \in \, $Ker$(\Phi)$. It follows that for $x, y \in A$ we have
$\Phi(x \circ y)$ equals $$\Phi((x - \Phi(x)) \circ (y - \Phi(y)) + \Phi(\Phi(x) \circ (y
- \Phi(y)) + \Phi(x  \circ \Phi(y)),$$
which is just $\Phi(x  \circ \Phi(y)) = \Phi(x) \circ \Phi(y) .$
Hence $\Phi$ is a weak* continuous Jordan $*$-homomorphism. 
Thus Ker$(\Phi)$ is a weak* closed selfadjoint two-sided ideal in $A^{**}$. By e.g.\ a variant of \cite[Fact 5.1.10]{Rod2}
(see also e.g.\  \cite[Theorem 3.25]{BWj}),
there
exists a central projection $p \in A^{**}$, with Ker$(\Phi) = p
A^{**}$. We have $$\Phi((1-p)a) = \Phi(a) - \Phi(p)
\Phi(a) = a, \qquad a \in A.$$ Thus $\Phi$ restricts to a surjective weak* continuous
faithful Jordan $*$-homomorphism from the $JW^*$-algebra $(1-p)
A^{**}$ onto $A$.
\end{proof}

\begin{theorem} \label{crpproj} {\rm  (\cite[Theorem 2.5]{BNp} and \cite[Corollary 4.18]{BNjp}) }\ Let $A$ be an 
approximately unital operator algebra (resp.\ Jordan operator algebra),
and $P : A \to A$ a completely contractive  completely real positive projection.
Then $P(A)$ is an approximately unital operator algebra (resp.\ Jordan operator algebra)  in the new product $P(xy)$
(resp.\ $P(x \circ y)$),
$P$ is still completely real positive as a map into the latter,
and $$P(P(a) P(b)) = P(aP(b)) = P(P(a) b)  \; \; , \qquad a, b \in A \, ,$$
(resp.\ $P(P(a) \circ P(b)) = P(a \circ P(b))$ for $a, b \in A$). 
If in addition $P(A)$ is a subalgebra (resp.\ Jordan subalgebra) of $A$ then $P$ is a conditional expectation 
(resp.\ Jordan conditional expectation).
\end{theorem}

The first step in the proof is to extend to the unitizations, using e.g.\ Theorem \ref{extn}.
This means that we may assume in the statement of the theorem that $A$ is unital and $P(1) = 1$, and we may then discard the
`completely real positive' as being automatic.  This is also the reason why 
$P$ is still completely real positive as a map into the `new algebra': this
follows from Proposition \ref{onladdn} because $P$ is (a restriction of) a unital complete contraction.

The above results are very satisfactory.   As we said, the present goal is to try to generalize such results to more general algebras, and 
so we now mostly consider  `contractive real positive projection variants'.   
That is, we will see what happens  e.g.\ to the parts of the last theorem
when we drop the  words `completely'.   Or in other
words, we move away from the 
{\em operator space} (in the sense of e.g.\ \cite{BLM,Pau}) case.  
This is a very significant step, indeed loosely speaking one of the main points of operator space theory is that being the 
`correct functional analysis' for many noncommutative problems, 
`it makes things work'.    We gave two justifications for attempting this step in the present paper
 above Theorem \ref{stinei1}.
 
Thus one would not really expect very strong results about real positive projections on general Jordan operator algebras
without a further hypothesis.  An important first illustration 
of the breakdown is that  $P(ab) \neq P(a) b$ in general for a unital operator algebra $A$ and contractive unital (hence real
positive) projection from $A$ onto a subalgebra containing $1_A$, and $a \in A, b \in P(A)$.
This is not even true in general if $A$ is commutative, which also rules out the Jordan variant $P(a \circ b) =  P(a) \circ b$ (see \cite[Corollary 3.6]{BM05}).
However we will see that the latter is true under further hypotheses, such as  if either $a$ or $b$ is in $\Delta(A)$.   

The next question, and one of the most important open questions here, 
is whether the range of a  real positive contractive projection $P$ on an approximately unital (Jordan) operator algebra is always again a Jordan operator algebra in the $P$-product?
Probably the answer is in the negative.   However we are able to prove this kind of result under various hypotheses, illustrated by some of the next several theorems.

 \begin{theorem} \label{cdsproj}  Let $A$ be an 
approximately unital operator algebra,
and $P : A \to A$ a contractive real positive projection.
Suppose that ${\rm Ker}(P)$ is
densely spanned by the real positive elements which it contains.
Then the range $P(A)$ is an
approximately unital operator algebra with product $P(x  y)$,
and $P$ is still real positive as a map into the latter
algebra.   
Also  $$P(a  b) = P(a  P(b)) = P(P(a)b) = P(P(a)  P(b)), \qquad a, b \in A.$$ 
If further $P(A)$ is a subalgebra of $A$ then $P$ is a homomorphism with respect to the $P$-product on its range.
\end{theorem}

 This  last result is a simpler special case 
of \cite[Theorem 4.10]{BNjp}.  
 The `densely spanned' hypothesis is a condition that is essentially always satisfied for
positive projections on $C^*$-algebras, so it is not unnatural.   

\begin{proposition} \label{delt}    Let $A$ be  a
(not necessarily approximately  unital)
 Jordan operator algebra,
and let  $P : A \to A$ be a contractive real positive projection.  The restriction of $P$ to the $JC^*$-algebra
$\Delta(A)$ is  positive and satisfies  $$P(\Delta(A)) = \Delta(P(A)) =  \Delta(A) \cap P(A).$$
\end{proposition}

\begin{proof}  Here  $\Delta(A) = A \cap \Delta(A^1)$, and
$\Delta(P(A))$ consists of the elements in $P(A)$ such that $x^* \in P(A)$ (so that $x \in \Delta(A)$,
and we can take the last involution to be the one in $\Delta(A)$).
The restriction of $P$ to the $JC^*$-algebra
$\Delta(A)$ is real positive, so it is positive and maps into $A \cap \Delta(A^1) = \Delta(A)$ by \cite[Lemma 2.7]{BNjp} and
the lines after \cite[Corollary 2.5]{BWj}.
 By the above,
$\Delta(P(A)) \subset \Delta(A) \cap P(A)$.   However
$\Delta(A) \cap P(A) \subset P(\Delta(A))$ since if $x \in \Delta(A)$ with $x = P(x)$, then $x \in  P(\Delta(A))$.
Finally, $P(\Delta(A))  \subset \Delta(P(A))$,
 since if $x \in \Delta(A)$ with $x = x^*$, then $P(x) = P(x)^* \in \Delta(P(A))$.
 \end{proof}

\begin{lemma} \label{bidjo} Suppose that $A$ is  an approximately  unital  Jordan
operator algebra, and $P : A \to A$ is a contractive real positive projection such that 
$P(A)$ is a Jordan operator algebra with  the $P$-product.  Then this Jordan algebra is approximately  unital,
and 
 $P^{**}(A^{**})$ is a unital  Jordan
operator algebra with the $P^{**}$-product.  \end{lemma}

\begin{proof}  Let $B$ be $P(A)$ with  the $P$-product. By basic functional analysis
$$P(A)^{\perp \perp}  = ( {\rm Ker}(P^*) )^\perp = (I-P^*)(A^*)^\perp={\rm Ker}((I-P^*)^*) = P^{**}(A^{**}),$$
since ${\rm Ker}(I-P^{**}) = P^{**}(A^{**})$.   Thus $P(A)^{**} = B^{**} \cong P^{**}(A^{**})$, via the bidual of the canonical inclusion $i : P(A) \to A$.  
Now $B^{**}$ is 
a Jordan
operator algebra with separately weak* continuous product $m^{**}(\zeta, \eta) = 
\lim_s \, \lim_t \, P(a_s \circ b_t)$, assuming that $a_s \to \zeta$ and $b_t \to  \eta$ weak* in $B^{**}$.
Here $a_s, b_t \in B$.  Now $i^{**}(P(a \circ b)) = P^{**}(i^{**}(a)  \circ i^{**}(b))$ for $a, b \in B$.
Hence $P^{**}(A^{**})$  is 
a  Jordan
operator algebra with product $$\lim_s  \, \lim_t \, i^{**}(P(a_s \circ b_t)) = \lim_s  \, \lim_t \, P^{**}(i^{**}(a_s) \circ i^{**}(b_t)) = P^{**}(i^{**}(\zeta) \circ i^{**}(\eta) ),$$ 
which is  the $P^{**}$-product.   These are weak* limits.  

By the last assertion of \cite[Lemma 4.9]{BNjp} applied in the bidual, $P^{**}(1)$ is an identity for the $P^{**}$-product.  So 
$B^{**}$ is 
a  unital Jordan
operator algebra, and hence $B$ is 
approximately unital by \cite[Lemma 2.6]{BWj}.  
 \end{proof}

 \begin{theorem} \label{i2} Let $A$ be  an approximately  unital  Jordan
operator algebra,
and let  $P : A \to A$ be a contractive real positive projection.
Then $\Delta(A) \cap P(A)$ is a $JC^*$-algebra  in the $P$-product.  If 
$P(A)$ is a Jordan operator algebra with  the $P$-product then 
  $$P(a \circ b) =
P(P(a) \circ b), $$   for
$a \in A$ and $b \in \Delta(A) \cap P(A)$; or for $a \in \Delta(A)$ and $b \in  P(A).$
In particular, if $A$ is unital then $P(1)$ is an identity for the $P$-product on $P(A)$.
 \end{theorem}

\begin{proof}   The restriction $E$ of $P$ to
$\Delta(A)$ is a positive contractive  projection by 
Proposition \ref{delt}, with range $\Delta(A) \cap P(A)$.  
Thus by Theorem  \ref{frs} we have that $\Delta(A) \cap P(A)$ is a $JC^*$-algebra  in the new product $P(x \circ y)$ and the old 
involution,
and $E$ is still positive as a map into the latter $JC^*$-algebra.   
The case of the displayed equation  when $a \in \Delta(A)$ and $b \in  P(A)$
may be found in  \cite[Lemma 4.9]{BNjp}. 

We next show that if $q$ is a projection in the the latter $JC^*$-algebra 
then   $P(qaq) = P(q P(a) q)$ and $P(a \circ q) =
P(P(a) \circ q)$ for all $a \in A$. If $q$ is an identity for $A$
then these assertions are trivial.  If not,
by Theorem \ref{extn}  we may assume that $A$ and $P$ are unital. 
So $q = P(q^2) \geq 0$.   
Claim: $P(q^n) = q$ for all $n \in \Ndb$.  In fact this is clear by the $C^*$-theory, but we give a 
short alternative proof. 
For $n = 1, 2$ this is clear.  By Theorem \ref{frs} 
$P(q^{n+1}) = P(q \circ P(q^n))$, the claim follows by induction.

By functional calculus we can approximate 
$q^{\frac{1}{n}}$ appropriately by polynomials $p_m(q)$ in $q$,
where $p_n$ has no constant term.  Then $P(p_m(q)) = p_m(1) \, q$
clearly.    If this approximation is done carefully one sees  
that $P(q^{\frac{1}{n}}) = q$ for all $n \in \Ndb$.

Let $Q = 
P^{**}$ and  $M = A^{**}$.
So 
$M$ is unital and weak* closed  and $Q$ is weak* continuous
and unital.     Note that $\Delta(M)$ is then weak* closed (since if $a_t, a_t^* \in A$ and $a_t \to \eta$ weak* 
then $a_t^* \to \eta^* \in A^{**} = M$ weak*).  Hence  so is  $Q(\Delta(M)) = \Delta(M) \cap P(M)$ weak* closed. 
Note that $\Delta(M)$ contains
the weak* closed algebra $N$ in $M$ generated by $q$ and $1$, a von Neumann algebra.  Since $Q(q^{\frac{1}{n}}) = q$ we have $Q(s) = q$ where $s$ is the support projection of $q$ 
in $M$.
Let $R$ be the restriction of $Q$ to  the unital Jordan operator algebra $sMs$.   Then $R$  is real positive.
By \cite[Lemma 2.8]{BNjp}  or the remark after it, we have $R(sas) = R(qsasq) = R(qaq)$ for all $a \in M$.
On the other hand, let us now view $Q$ as a map into $Q(M)$ with the $P$-product, which is a unital Jordan operator algebra
by Lemma \ref{bidjo}.  Then 
 we have $Q(sas) = Q(q Q(a) q)$ by Lemma \ref{ishe3}.
Hence $$P(qaq) = Q(qaq) = R(qaq) = R(sas) = Q(sas) = Q(q Q(a) q) = P(qP(a)q).$$

Next note that $1-q$ is a selfadjoint contraction with
$$P((1-q)^2) = 1 - 2q + P(q^2) = 1-q.$$ 
Thus $P(q^\perp aq^\perp) = P(q^\perp P(a)q^\perp)$ too, by the last paragraph.
By the argument at the end of the proof of 
Lemma \ref{ishe3} we deduce that 
$P(a \circ q) = P(a \circ P(q))$. 
   
 For the remaining  case of the displayed equation, 
we may assume that $A$ is weak* closed  and $P$ is weak* continuous (or else replace by $M$ and $Q$ above).  
This will work here  because $\Delta(A) \cap P(A) \subset \Delta(A^{**}) \cap P^{**}(A^{**})$.
By Theorem \ref{Sak},   $\Delta(A) \cap P(A)$ is a $JW^*$-algebra in the $P$-product.   
Hence $\Delta(A) \cap P(A) $ is spanned by the projections (in the $P$-product) which it contains (by essentially 
the same proof of the analogous fact for von Neumann algebras).  
Thus the second assertion yields the third (and centered) assertion.
\end{proof}

\begin{corollary} \label{ipscor}  Let $A$ be an 
approximately unital Jordan operator algebra,
and $P : A \to A$ a  real positive contractive projection with $P(A) \subset \Delta(A)$.
Then $P(A)$ is a $JC^*$-algebra in the $P$-product, and the restriction of $P$ to ${\rm joa}(P(A))$
is a Jordan $*$-homomorphism onto this  $JC^*$-algebra.
In this case $P$ is a Jordan conditional expectation with 
respect to the $P$-product:  $$P(a \circ P(b)) =P(P(a)  \circ P(b))$$ for $a, b$ in $A$.  \end{corollary}

\begin{proof}   By Proposition \ref{delt}, the restriction of $P$ to the $JC^*$-algebra $\Delta(A)$ is a real positive, hence positive, contractive projection
onto $P( \Delta(A)) = P(A) \subset \Delta(A)$.  
It is also $*$-linear.    
So  $P(A)$  is a $JC^*$-algebra in the $P$-product 
by Theorem \ref{frs}.  

 By   Remark 2  after  Theorem 4.10 in \cite{BNjp},  ${\rm Ker}(P) \cap \, {\rm joa}(P(A))$ is densely spanned by the real positive elements which it contains.
Hence $P$ is a Jordan homomorphism from ${\rm joa}(P(A))$  onto  $P(A)$ with the $P$-product,  by Theorem 4.10 in \cite{BNjp}.

That $P$ is a Jordan conditional expectation follows 
 from the last assertion of Corollary \ref{i2}, since $P(A)$ is selfadjoint.  
\end{proof}

The following result of independent interest  is contained in  
Theorem 3.5 of  \cite{Rod3}.   Indeed, the
first paragraph in the lemma is the equivalence of 
(4) and (8) in that paper (taking $u=1$ in (8)).   Actually
that paper  is a revision of \cite{R98}, and the first paragraph of
our lemma may also be seen from  Lemma 1.1 and the equivalence of (i) and (iv)
in Corollary 2.7 of \cite{R98}.     
 However we do not see some of  it stated there in exactly
the form below, except in the von Neumann algebra case,
and  so we thank Angel Rodr\'iguez Palacios for allowing us to
give here a direct proof avoiding nonassociative algebra.   By a Banach algebra we mean an associative algebra with a complete norm that is submultiplicative:  
$\| x y \| \leq \| x \| \| y \|$. 

\begin{lemma}  \label{iskk}
Let $B$ be a $C^*$-algebra with identity $1$.
The possible Banach algebra products on $B$ with identity $1$
are all $C^*$-algebra products.  They are in a bijective 
correspondence with the central projections in the multiplier 
algebra of the commutator ideal of $B$.

In addition, any Banach algebra product $m$ on a unital
$JC^*$-algebra is a $C^*$-algebra product
such that $\frac{1}{2}(m(x,y) + m(y,x))$  
is the original $JC^*$-algebra product.
\end{lemma}

 \begin{proof} Let $J$ be the commutator ideal of $B$,
and let $z$ be the
support projection of $J$ in $B^{**}$.
Suppose that $e$ is a central projection in $M(J)$.
Then $ex = xe$ for all $x \in J$, hence for for all $x \in J^{**}$.
For $x \in B^{**}$ we have $ex = e(z + z^\perp) x = e zx = zx e$.
Similarly $xe = zx e$, so that $e$ is central in $B^{**}$.
For $x, y \in B$ we have
$$exy + e^\perp yx = e(xy -yx) + yx \in J + B \subset B.$$
Then it is easy to check that
$m(x,y) = exy + e^\perp yx$ defines a $C^*$-algebra 
product on $B$ with identity $1$.  

If $e, f$ are central projections in $M(J)$
such that $exy + e^\perp yx = fxy + f^\perp yx$ for all $x, y \in A$,
then $$e(xy-yx) = 
exy + e^\perp yx - yx = fxy + f^\perp yx - yx = f(xy-yx).$$
Hence $e = f$.

Suppose that $m$ 
is a Banach algebra product on a $JC^*$-algebra $B$ with identity $1$.
The hermitian elements
in this Banach algebra are just 
the selfadjoint elements
in the $JC^*$-algebra $B$, since hermitians in a 
unital Banach algebra
depend only on the norm and the identity element.  These hermitians 
span $B$.  Thus $B$ with product $m$ is
a unital $C^*$-algebra by the Vidav-Palmer theorem
(see e.g.\ \cite[Theorem 2.3.32]{Rod}).   
Then  $\frac{1}{2}(m(x,y) + m(y,x))$ 
is a $JC^*$-algebra product on $B$ which must be  the original $JC^*$-algebra product
by the Banach-Stone theorem (e.g.\ by Theorem \ref{AS}).

Again suppose that $C$ is a $C^*$-algebra and write $C$ for $B$ with product $m$.  The identity map is a
surjective unital isometry from $C$ onto $B$, thus
is a Jordan $*$-isomorphism $\theta$ by Kadison's Banach-Stone  theorem for
$C^*$-algebras \cite{Kad}.  By another result of Kadison 
in the same paper, there is a central projection
$e \in C^{**}$ such that
 if $f = \theta^{**}(e)$ then 
$e \theta^{**}(\cdot)$ is a $*$-homomorphism and
$f^\perp \theta^{**}(\cdot)$ is a $*$-anti-homomorphism.
In our setting $e = f$ and  $$e \theta^{**}(m(a,b))
= e \theta^{**}(a) \theta^{**}(b), \; \; \; 
e^\perp \theta^{**}(m(a,b))
= e^\perp \theta^{**}(b) \theta^{**}(a) , a, b \in B.$$
That is, 
$$e m(a,b) = e ab, \; \; \;
e^\perp m(a,b) = e^\perp ba , a, b \in B,$$
in the usual product on $B^{**}$.
Thus $m(a,b) = e ab + e^\perp ba$ for $a, b \in B$.    (The facts in this paragraph also follow from 
Theorem 3.2 of \cite{Rod3}.)
 
Note that for $x, y \in B$ we have
$$exy + e^\perp yx = e(xy -yx+yx) + e^\perp yx = ez (xy -yx) + yx =
ezxy + (ez)^\perp yx .$$ 
Thus if we replace $e$ by $ez$ we may assume that 
$e$ is a central projection in $J^{**} = zB^{**}$.
We have 
$$e(xy -yx) = (exy + e^\perp yx) -yx \in B \cap J^{**} = J,
\qquad x, y \in J.$$
 Thus $e \in M(J)$ and is central there.    
\end{proof}

\begin{remark}  One may ask if $B$ is a selfadjoint Jordan subalgebra of a unital $C^*$-algebra $D$ containing $1$,
and if  $B$ has a Banach algebra product with identity $1$, or equivalently (by the last result) a
 $C^*$-algebra product with identity $1$, then is $B$ an (associative) subalgebra of $D$?
 The answer is in the negative:  consider $B = \{  a \oplus a^\intercal  : a \in M_2 \} \subset M_4$.
 \end{remark} 

\begin{theorem} \label{ipscor2}  Let $A$ be an 
approximately unital  (associative) operator algebra,
and $P : A \to A$ a  real positive contractive projection with $P(A) \subset \Delta(A)$.  
The $P$-product on $P(A)$ is associative if and only if 
$$P(P(a)P(b)) = P(aP(b)) = P(P(a)b), \qquad a, b \in A.$$  In this case 
  $P(A)$ is a $C^*$-algebra with respect to the $P$-product,
  and $P$ viewed as a map into this latter $C^*$-algebra is real completely positive and completely contractive.
 \end{theorem}

\begin{proof}    If the centered equation holds then $$P(P(P(a)P(b))P(c)) = P(P(a)P(b)P(c)) = P(P(a) P(P(b)P(c))).$$   So the $P$-product on $P(A)$ is associative. 

Suppose that the $P$-product on $P(A)$ is associative.  By Theorem \ref{extn} we may assume that $A$ is
unital and $P(1) = 1$. 
 (Note that if $A$ is
nonunital or $P(1) \neq 1$ then $\tilde{P}(A^1) \subset \Delta(A^1)$; and if the new product on $P(A)$ is associative,
then so is  the new product on $\tilde{P}(A^1)$.  
Here $\tilde{P}$ is the extension to $A^1$ from Theorem \ref{extn}.)    
Write $B$ for $P(A)$ in the $P$-product.
This is a unital Banach algebra.  By Corollary \ref{ipscor}  $P(A)$ is also a unital $JC^*$-algebra.  By the proof of Lemma  \ref{iskk},  $B$ is a unital $C^*$-algebra. By considering $P^{**}$ and $A^{**}$ we 
may suppose that $B$ is a $W^*$-algebra.

Write $\cdot$ for the $P$-product.  Claim:  $a \cdot P(x) =  P(ax)$ for all $a \in B, x \in A$.  
To prove this, as in the proof of \cite[Lemma 3.2]{BM05a} it suffices to show that  
 \begin{equation} \label{pperp} p^\perp \cdot P(p x) \; = \; 0 ,  \qquad x \in X , \end{equation}
for all orthogonal projections $p \in B$.  For then we get
$p \cdot P(x)  =  P(px)$  as in that cited proof, and since $B$ is densely spanned by its
projections (as is any $W^*$-algebra), we conclude that $a \cdot P(x) =  P(ax)$ for all $a \in B$.  

To prove (\ref{pperp}), we adjust the argument
in the last cited proof.
Let $$y = P(px) + t P(p^\perp P(px)) = P(px) + t p^\perp \cdot P(px) , \qquad 
t \in \Rdb .$$   By associativity, $y = 
P(px + p^\perp P(p^\perp P(px)))$.   
Now $$\| y \|^2 \leq \| px + t p^\perp P(p^\perp P(px) ) \|^2,$$ which as in the last cited proof is
dominated by 
$\Vert p x \Vert^2 + t^2 \Vert P(p^\perp P(px)) \Vert^2.$
 On the other hand, again writing
$\cdot$ for the $P$-product, we have
$$p^\perp  \cdot y = p^\perp \cdot P(px) + t p^\perp \cdot P(px)
= (1+t) p^\perp \cdot P(px) .$$
It follows that
$$(1+t)^2 \Vert p^\perp  \cdot P(px) \Vert^2 \leq
\Vert p x \Vert^2 + t^2 \Vert p^\perp \cdot P(px) \Vert^2.$$
This implies that $2t \Vert p^\perp \cdot P (px) \Vert^2 \leq 
\Vert p x \Vert^2$
for all $t > 0$, so that $p^\perp \cdot P(px) = 0$.  

The Claim says that $P$ is a left $B$-module map;
similarly it is a right $B$-module map, giving the 
last assertion of our statement.
It follows by a standard trick similar to e.g.\ 1.2.6 in 
\cite{BLM} (see e.g.\
\cite[Proposition 8.6]{Pau}) 
 that $P$ is completely contractive.
Since $P(1) = 1$, $P$ is real completely positive by 
(the matrix versions of) considerations in the paragraph 
before Corollary 2.2 from \cite{BNjp}.  
\end{proof}

\begin{remark}   By 
 the proof above,  one may relax the associativity condition in the last theorem 
to: $p \cdot (p \cdot b) = p \cdot b$ for $b \in P(A)$
and for $p \in P(A)$ which are projections  in the $P$-product.

We also remark that a real positive completely contractive projection
need not be real completely positive.  Indeed   
$P(a,b) = (a , a^\intercal /2 )$ on $M_2 \oplus M_2$ is a counterexample
(we thank R. R. Smith for this).  Note that this completely contractive positive projection
does not even satisfy the Kadison-Schwarz inequality $P(a^*a) \leq
P(a)^* P(a)$. 

The $C^*$-algebra structure on $P(A)$ in the conclusion of the theorem 
may not induce the operator space structure on $P(A)$ inherited from $A$.  To see this consider 
$P(a,b) = (a , a^\intercal )$ on $M_2 \oplus M_2$.
\end{remark}

The following result  is inspired by the selfadjoint case due to Effros and St{\o}rmer (see e.g.\ \cite[Lemma 1.4]{ES}).
If $A$ is a unital operator algebra, and $P$ is  a unital contractive or
completely contractive projection
on $A$, define
 $$N = \{ x \in A : P(xy) = P(yx) = 0 \; \text{for all} \; y \in A \}.$$
If $A$ or $P$ is not unital, but $P$ is also real positive,
then we may extend $P$ by Theorem \ref{extn} to a unital contractive projection
on $A^1$, where $A^1$ is a unitization with $A^1 \neq A$, and 
set  $$N = \{ x \in A : P(xy) = P(yx) = 0 \; \text{for all} \; y \in A^1 \}.$$ 
   Then $N$ is clearly a closed ideal in $A$, and is also a subspace of Ker$(P)$.  Define
$$B = \{ x \in A : P(xy) = P(P(x) P(y)) \, {\rm and} \, P(yx)  = P(P(y) P(x)) \; \text{for all} \; y \in A \}.$$  Note that $N \subset B$ since e.g.\
 if $x \in N \subset {\rm Ker}(P)$
then $P(xy) = 0 = P(P(x) P(y))$ for all $y \in A$.
Note too that $1 \in B$ if $A$ is unital and $P(1) = 1$.

\begin{theorem} \label{AB} If $P$ is a  real positive contractive  projection on an approximately unital operator algebra $A$, and 
$N, B$ are defined as above, 
then  $P(A) \subset B$ if and only if
$$P(P(a) b) = P(P(a) P(b)) = P(a P(b))  \; \text{for all} \; a , b \in A.$$
That is, if and only if $P$ is a conditional expectation onto $P(A)$
 with respect to
the $P$-product.  This is also equivalent to $B = P(A) + N$.
If these hold then  $P(A)$
with the $P$-product is isometrically isomorphic
 to an operator algebra, $B$  is a subalgebra of $A$ containing
$P(A)$, and $P$ is a homomorphism from $B$ onto  $P(A)$
with the $P$-product.
\end{theorem}
\begin{proof}  The first  `if and only if' follows from the definition of 
$B$.   This is also equivalent to $B = P(A) + N$,
since $N \subset B$, and  if $P(A) \subset B$ and
$a \in B$ then $$P((a - P(a))b) = P(P(a)P(b)) -
P(P(a) b) = 0 , \qquad b \in A .$$  Similarly, $P(b(a - P(a))) =0$.  
Hence $a = a - P(a) + P(a) \in N + P(A)$.  

In this case $B$ is a subalgebra of $A$ since e.g.\ if $a, b \in B, c \in A$
then
$$P(abc) = P(P(a)P(bc)) = P(P(a)P(P(b)P(c))) = P(P(a) P(bP(c)),$$
and by the centered equation in the theorem statement,
 $$P(P(ab) P(c)) = P(ab P(c)) = P(P(a) P(bP(c)).$$
  Thus also the $P$-product on $P(A)$ is associative. 
Since $B/N \cong P(A)$ as in the proof of Lemma 4.4 in \cite{BNjp}, and quotients of operator algebras are  operator algebras
\cite[Proposition 2.3.4]{BLM},  we see  that  $P(A)$
with the $P$-product is isometrically isomorphic
 to an operator algebra.  The last assertion again follows from the definition of $B$. 
\end{proof} 

\begin{remark}  Note that $N = {\rm Ker}(P)$ 
if and only if  
${\rm Ker}(P)$ is an ideal.  The latter
holds (by the associative algebra variant of 
Lemma 4.4 in \cite{BNjp}) if and only if   $B = A$, and then all of the conclusions of the last theorem hold.

If $P$ is real completely positive and completely contractive then $$P(P(a) b) = P(P(a) P(b)) = P(a P(b)), \qquad a, b \in A,$$  as is proved in \cite[Section 2]{BNp}, so that the conclusions of the last theorem hold. 
\end{remark}

\begin{corollary} \label{cdsproj11}  If   $A$ is an  approximately unital  (associative) operator algebra and $P : A\to A$ is  contractive and real positive, and is a 
conditional expectation  onto $P(A)$ equipped with the product $P(ab)$
(that is, if  $P(P(a) b) = P(P(a) P(b)) = P(a P(b))$ for all $a , b \in A$), then 
$P(A)$ is an operator algebra with this product.
\end{corollary} 

This corollary shows that $P$ being a conditional expectation  for the $P$-product implies that the $P$-product is an operator algebra product.  The converse is false as we observed below Theorem  \ref{Sak}.    
 
It is natural to ask if similar results hold for Jordan operator algebras, and in particular 
does $P$ being a conditional expectation  for the Jordan $P$-product   imply  that $P(A)$ is a  Jordan operator algebras
in the $P$-product.   The latter is an interesting question that we have not yet been able to solve.  
  If one goes through the proof above 
 with a (without loss of generality) contractive  unital projection $P$ on a unital Jordan operator algebra $A$,
 one defines $$N = \{ x \in A : P(x \circ y)  = 0 \; \text{for all} \; y \in A \}.$$ 
    Then $N$ is a subspace of Ker$(P)$, and is a  closed `Jordan ideal' in 
$$B = \{ x \in A : P(x \circ y) = P(P(x) \circ P(y)) \; \text{for all} \; y \in A \}.$$  
Again $B = P(A) + N$, and $P$ restricts to  a  contractive  unital projection and `Jordan homomorphism' from 
$B$ onto $P(A)$ with kernel $N$.    So $P(A) \cong B/N$ isometrically and Jordan isomorphically.   However, two issues arise.
First it is not clear that $B$ is actually a Jordan (operator) algebra.   Second, even if it were we do not known that
a quotient of a Jordan operator algebra by a closed Jordan ideal is a Jordan operator algebra.  This is a big open problem in the 
subject.  If it is false then this suggests that perhaps the study of closed Jordan subalgebras of such quotients may be an interesting
direction of research.    The usual proof that the quotient of an  operator algebra by a closed ideal is an operator algebra, uses the so-called 
BRS characterization of operator algebras (see \cite[Theorem 2.3.2 and Proposition 2.3.4]{BLM}).  There is a somewhat similar theorem for Jordan operator algebras, namely \cite[Theorem 2.1]{BWj}.
If one attempts the natural proof (analogous to the operator algebra quotient proof) one sees that it works for quotients of a Jordan operator algebra $B$ by a closed Jordan ideal $N$
such that $B$ is contained in an operator algebra $A$ with $B/N \subset A/[ANA]$   completely isometrically.   Here $[ANA]$ is the closed 
associative ideal in $A$ generated by $N$.    That is, in this case $B/N$ is a Jordan operator algebra.    We do not know unfortunately when 
$B/N \subset A/ [ANA]$.   This would require $B \cap [ANA] = N$ at the very least.  Perhaps there is a clever choice of 
$A$ that will do this, perhaps even a $C^*$-algebra. 

In particular, in the situation in the last paragraph, but  now assuming in addition that  $N = (0)$, one may ask
if $B = P(A)$  a Jordan subalgebra of $A$, 
and if  $P$ is a  Jordan conditional expectation?

\begin{corollary} \label{ipscor3}  Let $A$ be an
approximately unital  (associative) operator algebra,
and $P : A \to A$ a  real positive contractive projection onto an
associative subalgebra of $A$, with $P(A) \subset \Delta(A)$.   Then
 $P$ is completely contractive and real completely positive as a map into $A$,
 $D = P(A)$ is a $C^*$-algebra in the $P$-product, and $P$ is a $D$-bimodule map:
   $P(P(a)P(b)) = P(aP(b)) = P(P(a)b)$ for all $a, b \in A.$
 \end{corollary}
 
The last result is a corollary of Theorem \ref{ipscor2}.

We now turn to bicontractive projections.  
The following result   shows what happens in the case of selfadjoint Jordan operator algebras ($JC^*$-algebras). 
It is the `solution to the bicontractive and symmetric projection problems' for $JC^*$-algebras, essentially 
due to deep work of Friedman and Russo, and St{\o}rmer \cite{FRAD, FRceJ, ST82}. 
We recall that  $P$ is {\em bicontractive} if $\| P \|, \| I - P \|$ are contractions, and {\em symmetric}  if $\|I - 2P \| \leq 1$. 
Some of this hinges on Harris's Banach--Stone type theorem
for $J^*$-algebras \cite{Har2}.  The following 
is essentially very well known (see the references above), but 
we do not know of a reference besides the work which we are surveying (see \cite[Theorem 5.1]{BNj}) which has all of these assertions, or is in the formulation  we give:  

\begin{theorem} \label{goq}  If $P : A \to A$ is a    
projection on a $JC^*$-algebra $A$ then 
$P$ is bicontractive  if and only if  $P$ is symmetric.  Moreover  $P$ is bicontractive and 
positive  if and only if  there exists a central projection $q \in M(A)$ (indeed $q \circ a = qaq \in A$ for all $a \in A$) such that
 $P = 0$ on $q^\perp A q^\perp$, and there exists a Jordan $*$-automorphism $\theta$ of $qAq$   of period 2 (i.e.\ $\theta \circ \theta = I$)
so that $P = \frac{1}{2}(I + \theta)$ on $qAq$.     Finally, $P(A)$ is a $JC^*$-subalgebra of $A$, and 
$P$ is a Jordan conditional expectation.  
\end{theorem}  

We remark that subsequently the variant of this theorem with $JB^*$- instead of $JC^*$-algebras has been proved in
\cite[Theorem 5.7]{CR20}.

We have a complete characterization of symmetric real positive projections on Jordan operator algebras, which relies on the very recent
Banach-Stone theorem 
\ref{bs1} above.   The point is that if $P$ is symmetric then it is very easy to show that $v = I - 2P$ is a  surjective isometric
isomorphism of $A$ onto $A$ which has period 2 (that is, $v \circ v = I_A$).    Then one applies the Banach-Stone theorem 
\ref{bs1} above.     This is the main ingredient in the following result from \cite{BNjp}: 

\begin{theorem} \label{symm} Let $A$ be an approximately unital  (Jordan) operator algebra,
and $P : A \to A$ a symmetric  real positive projection.
Then the range of $P$  is an approximately unital 
Jordan subalgebra of $A$ and 
$P$ is a Jordan conditional expectation.  Moreover, $P^{**}(1) = q$ is a projection in  $JM(A)$.

Set $D  = q Aq$,  the  hereditary subalgebra (`corner')  of $A$ supported by $q$,   which contains $P(A)$.
There exists a period $2$ surjective  
isometric Jordan homomorphism $\pi : D \to D$, such that 
   $$P =  \frac{1}{2} (I + \pi) \; \; \;
{\rm on} \; D,$$ 
and $P$ is zero on the `other three corners' of $A$
(that is, on $q^\perp A q^\perp + q^\perp A  q+ qA q^\perp$).   
\end{theorem}  

The converse is true too, such an expression $P =  \frac{1}{2} (I + \pi)$ is a symmetric real positive projection.

This is very close to the  `classical' selfadjoint  characterization in 
Theorem \ref{goq} above.  The main difference is $q$ need not be central, and $P$  symmetric 
is not equivalent to $P$ bicontractive.

The form of the {\em bicontractive projection problem} that evolved in \cite{BNp,BNjp}
asks for  conditions on a bicontractive real positive projection $P : A \to A$ so that
$P(A)$ is a  subalgebra of $A$?   This is not always true, as \cite[Corollary 4.8]{BNp} shows.

 In \cite[Section 4]{BNp} and the start of \cite[Section 6]{BNjp}  we gave a three step reduction that reduces the  bicontractive projection problem to the case of a bicontractive  projection $P : A \to A$ with $P(1) = 1$ and $A$ is generated
by $P(A)$.   We will omit the details here, although we note that some of  the ingredients in this
reduction are taking the bidual, and then observing that for a  bicontractive real positive projection
on a unital operator algebra $A$, $P(1)$ is a projection.   Also we use some facts about Jordan hereditary subalgebras.  

Part of the following result is \cite[Theorem 6.3]{BNjp}.

\begin{theorem} \label{bicr} 
Let $A$ be a unital (Jordan)  operator algebra, and let $D$ be the elements in ${\rm Ker}( P)$ that are also in the 
closed Jordan subalgebra generated by $P(A)$. 
 If $P:A \rightarrow A$ is a bicontractive  unital projection on $A$,   and if $D$  is
densely spanned by the real positive elements which it contains, or if $P(A) \subset \Delta(A)$, then $P(A)$ is a Jordan  subalgebra of $A$.  
\end{theorem}  

\begin{proof}
We will just prove the result with hypothesis
$P(A) \subset \Delta(A)$ here, for the other see \cite[Theorem 6.3]{BNjp}.   If $P(A) \subset \Delta(A)$ then as in the proof of 
Corollary \ref{ipscor},  the restriction of $P$ to the $JC^*$-algebra $\Delta(A)$ is a real positive, hence positive, bicontractive projection
onto $P( \Delta(A)) = P(A) \subset \Delta(A) = P(A)$.   By Theorem \ref{goq}, $P(A) $ is a 
Jordan  subalgebra of $\Delta(A)$, hence of $A$.  \end{proof}

   (The hypotheses in the last theorem are conditions that are always satisfied for
positive projections on $C^*$-algebras, so they are not unnatural.)

Finally, we remark that much of  Section \ref{cp}  was concerned with which results from  \cite{BNp} 
still have variants valid for {\em contractive real positive} projections on Jordan operator algebras.   If one uses 
 {\em completely contractive completely  real positive} projections then essentially everything in  \cite{BNp}
is valid for Jordan operator algebras, as is observed in \cite{BNjp}.

 \section{What are conditional expectations?} \label{wace} 
 
 This section may be viewed as an appendix for nonexperts, giving some basic insights into the conditional expectation property in 
 Section \ref{cp}, and into some of the advanced considerations needed for the discussion in Section \ref{BLvv}.  
 Conditional expectations are no doubt also addressed in other articles in this conference proceedings, probably also 
 with some discussion of their history.
 Thus there may be a very small amount of overlap here at the beginning of the present section.   We believe that these other articles are located 
 in the Banach lattice setting and are concerned with applications there.
 We will not focus on the lattice aspect at all, and our techniques are quite different. 
  
 The classical or `probabilistic' conditional expectation may be taken to refer to 
 a hugely important construction on a probability 
 measure space $(K,\A,\mu)$ induced by a choice of a sub-$\sigma$-algebra $\B$ of the $\sigma$-algebra $\A$ on the set $K$. 
 In this case one may identify $L^p(K,\B,\mu)$ isometrically with a subspace of $L^p(K,\A,\mu)$.   Moreover dualizing these embeddings
 gives contractive positive unital projections $E_{\B}$ of $L^p(K,\A,\mu)$ onto the copy of $L^p(K,\B,\mu)$ (to get the expectation onto $L^1$ one takes
 the dual in the weak* topology of the $L^\infty$ inclusions).  
 These satisfy a list of beautiful 
 properties often attributed to Kolmogorov (a great account of this list may be found e.g.\ in the Wikipedia article on 
 conditional expectations).    Writing $E_{\B}$ as $E$, the most notable of these is
 that $$\int \, E (f)  \, d \mu = \int \, f   \, d \mu$$ (which in the quantum variant becomes the `trace preserving' property).
 However this list also includes `positivity': $E(f) \geq 0$ if $f \geq 0$, `contractivity' ($|Ef | \leq E(|f|)$ which implies
 $\| E \| \leq 1$), the projection property $E \circ E = E$, and that 
  $E(1) = 1$.
 It  also includes the important `module' property $$E(fg) = E(f) g, \qquad g \in L^\infty(K,\B,\mu),$$
 which defined what we called `conditional expectation' in Section \ref{cp}.  
 It also has important continuity properties, like being weak* continuous on $L^\infty$ (the latter is clear 
 because as we said it is the dual 
 of a map on $L^1$).
 Moreover $E_{\B}$ has a fundamental probabilistic interpretation.  We will not rehearse this here since it is so well
 known.   Indeed 
this   interpretation is ubiquitous in scientific disciplines. It has been said 
 that  conditional expectations are the starting point of modern probability theory.    
 
 Most of this
 is still true if we start to mildly relax the condition that $\mu$ is a finite measure (although now $1 \notin L^1$).    Historically, mathematical
 analysts tried to successively weaken the measure theoretic 
 requirements on $\mu$ while still preserving a satisfactory theory of expectation.  
  At some point beyond so-called `localizable measures' (we warn the reader that there is  ambiguity 
 in the measure theory literature concerning this term) things break down and become pathological.
  It is interesting that when one tries to identify this point of breakdown, it seems  to corroborate 
 how perfectly von Neumann algebras capture the essence of these concepts.  Namely, the breakdown occurs very slightly beyond 
 the class of measures for which $L^\infty(K,\mu)$ is a von Neumann algebra.   And in the latter case the rich 
 theory of von Neumann algebraic conditional 
 expectations applies.

  Notice that $L^\infty(K,\B,\mu)$ is a von Neumann subalgebra of $L^\infty(K,\A,\mu)$.  Indeed this may be viewed as the 
  weak* closure of the  $\B$-simple functions; or equivalently, the weak* closed subalgebra of $L^\infty(K,\A,\mu)$ generated by the 
  (characteristic functions of the) sets in $\B$.    This is an easy exercise.
  
 Conversely, any von Neumann subalgebra $D$ of $L^\infty(K,\A,\mu)$ is of this form.  Since we are not aware of a source in the literature
 we sketch a simple argument that $D = L^\infty(K,\B,\mu)$.     Let $\B = \{ B \in \A : \chi_B \in D \}$.
It is an exercise to check that $\B$ is a $\sigma$-algebra.   For example it is an algebra because
$\chi_{B_1} \, \chi_{B_2} = \chi_{B_1 \cap B_2}$.   Then if $B_1, B_2, \cdots$ are disjoint sets in
$\B$, let $(p_n)$ be the corresponding mutually orthogonal projections in $D$, let $F = \cup_n \, B_n$, and let $p = \sup_n \, p_n \in D$.
Then there exists a set $B \in \A$ with $p = \chi_B$ $\mu$-a.e..   Since $p_n p = p_n$ the set $B_n \setminus (B \cap B_n)$ is
$\mu$-null.   Thus we may assume that $B_n \subset B$  for all $n$, so that $F \subset B$.  On the other hand if $E \cap B_n = \emptyset$ for all
$n$, then $\chi_E \, p_n = 0$ for all $n$, so that $\chi_E \, p = \chi_{E \cap B} =  0$ in $M$.  We
conclude that $\chi_F = p \in D$.  Thus $F \in \B$, so that  $\B$ is a $\sigma$-algebra.
If $B \in \B$ then $\chi_B \in D$ by definition.  Conversely, for any projection $p$ in $D$  there exists a set $B \in \A$ with $p = \chi_B$ $\mu$-a.e.,
so that $B \in \B$ and $\chi_B \in  L^\infty(K,\B,\mu)$.
Using the fact that $D$ and $L^\infty(K,\B,\mu)$ are both generated by their projections, we conclude that $D = L^\infty(K,\B,\mu)$.

In the discussion below we assume that $\mu$ is a probability measure for simplicity.
It follows that for any von Neumann subalgebra $D$ of $L^\infty(K,\A,\mu)$, there 
exists a canonical sub-$\sigma$-algebra $\B$ of $\A$ and a canonical contractive projection $E_{\B}$ from $L^\infty(K,\A,\mu)$ onto $D$.
We call this the 
{\em probabilistic conditional expectation}.   It is, as we said,  `trace preserving': $\int \, E_{\B}(f)  \, d \mu = \int \, f   \, d \mu$ for $f \in L^\infty(K,\A,\mu)$.

Some of the founders of modern probability theory and their students tried to characterize conditional expectations amongst the 
idempotent maps (i.e.\ projections) on $L^p(K,\A,\mu)$ (particularly in the case $p = 1$).    
See e.g.\ \cite{DR,Lacey} and references therein.  There were  early characterizations due
to Moy \cite{Moy} who characterizes conditional expectations in terms of  operators
on the positive measurable functions, and on $L^p$, obtaining 
particularly nice results for $L^1$.   Later Douglas  \cite{Doug}, Ando, Lacey  and Bernau (see e.g.\ \cite{Lacey}), and others refined these results.  
It follows from this work that there are bijective correspondences between  weak* continuous unital contractive projections  $P$ 
from $L^\infty(K,\A,\mu)$ onto a von Neumann subalgebra  $D$, 
and  density functions $h \in L^1(K,\A,\mu)_+$ with $E_{\B}(h) = 1$ (such $h$ is called a `weight function' for $p$).
The correspondence is given by $P(x) = E_{\B}(hx)$, and $h$ is the density of the normal state $x \mapsto \int \, P(x)  \, d \mu$ (which
may be written as $P_*(1)$, viewing $1 \in L^1$).   If $\int \, P(f)  \, d \mu= \int f  \, d \mu$ for $f \in L^\infty$ then 
one sees that $\int \, h f  = \int f$ for such $f$, which forces $h = 1$ and $P = E_{\B}$.  

Summarizing the above discussion: we have seen at least in the setting above, the weak* continuous unital contractive projections $P$
from $L^\infty$ onto a  von Neumann subalgebra $D$ 
are simply  the `weightings' of the  probabilistic conditional expectation $E_{\B}$, by the weights $h$ above.     That these projections have the  important `module' 
property $P(fg) = P(f) g$ for $g \in D$ (which defined what we called `conditional expectation' in Section \ref{cp}), may be viewed in this picture $P(x) = E_{\B}(hx)$ as coming immediately
 from the fact that 
 the  probabilistic conditional expectation $E_{\B}$ has this property.     The probabilistic conditional expectation
is characterized among all the weak* continuous unital contractive projections  $P$ onto $D$ by the 
`trace-preserving' condition $\int \, P(f) = \int f$.   

So all weak* continuous unital contractive projections $P$
from $L^\infty$ onto a  von Neumann subalgebra `are' weighted probabilistic conditional expectations.
In fact if $P$ is faithful then 
 it is a probabilistic conditional expectation, the probabilistic conditional expectation associated with a measure which we noew describe.  
 Indeed one may play the above game  in reverse.  Suppose that 
we are given a  weak* continuous unital contractive projection $P$ from  $L^\infty(K,\A,\mu)$ onto $D$.
Then $f \mapsto  \int \, P(f)  \, d \mu$ 
is a normal state 
of $L^\infty(K,\A,\mu)$, which is simply integration against  a probability measure $d \nu = h \, d \mu$ on $(K,\A)$.
If $P$ is faithful then $h$ has full support and $L^\infty(K,\A,\mu) = L^\infty(K,\A,\nu)$.  
The measure space $(K, \A,\nu)$  can 
replace the role of $(K,\A,\mu)$ in the above discussion, to produce a
probabilistic conditional expectation associated with $\nu$.   If there are no issues with the 
`support' (that is, if $h$ has full support), then one can see that this probabilistic conditional expectation is  $P$.

\begin{remark}  One can somewhat generalize the discussion beyond the case that $D$ is a subalgebra, to try to link up with the generalized `conditional expectations' discussed 
in Section \ref{cp} with respect to the Choi-Effros product.   For example,
 Douglas showed that the range of a contractive projection $P$ on $L^1$  is isometric to an $L^1$ space  \cite{Doug}, and investigated
 the relation of such $P$ to the probabilistic conditional expectation.   
Let us give a quick proof of this range assertion.   First assume that  $\int \, P(f)  = \int f$ for $f \in L^1$.
(The latter is a much weaker condition than the condition $\int \, P^*(f) = \int f$  or $f \in L^\infty$ considered above, indeed  $\int \, P(f) = \int f$ for $f \in L^1$ simply says
that $P^*$ is unital on $L^\infty$, which is satisfied by all the projections in the earlier discussion.) 
Indeed note that in this case $P^*$ is a contractive unital projection, and 
$$({\rm Ran}(P))^* \cong A^*/{\rm Ran}(P)^{\perp} = A^*/{\rm Ker}(P^*) \cong {\rm Ran}(P^*).$$
By the Choi and Effros result mentioned in the introduction, the range of $P^*$ on $L^\infty$ is a (commutative) $C^*$-algebra in the $P^*$-product.
Since it is weak* closed it is a von Neumann algebra (by a result like
Theorem \ref{Sak} if necessary).   By the uniqueness of von Neumann algebra
preduals, ${\rm Ran}(P)$ is an $L^1$ space.   There is a similar proof in the general case: by a result of Youngson
\cite[Theorem 4.4.9]{BLM} which generalizes the Choi-Effros result used above, $P^*(xy^*z)$ is a (commutative) TRO product on  ${\rm Ran}(P^*)$.     As e.g.\ in the proof of \cite[Theorem 4.4.9]{BLM}, 
any extreme point of the ball 
in this TRO is `unitary', and  any TRO with a unitary is isometric to a $C^*$-algebra.   So again as above it is a von Neumann algebra and 
${\rm Ran}(P)$ is an $L^1$ space.   See \cite{NgOz} (and the Kirchberg result discussed and cited there) for a noncommutative generalization. 
\end{remark} 

Let us now consider relaxing the condition that $\mu$ is a probability measure in the discussions above.   We said earlier that there are serious pathologies 
for the most general kinds of measures, and suggested to simply consider measures with 
$L^\infty(K,\A,\mu)$  a von Neumann algebra.   Going one step further, replace $L^\infty(K,\A,\mu)$  by 
a possibly noncommutative von Neumann algebra $M$.   The measure $\mu$, and associated integral, will be replaced by a certain `trace' or `weight' $\nu$.
Now we are in the setting of von Neumann algebraic conditional expectations, which gets into the deep taxonomy of von Neumann algebras,
and the difficult theory of noncommutative integration \cite{Tak2,LabG}.  
The case that we have dealt with above corresponds to the class of so-called `finite' von Neumann algebras, where there exists a faithful normal tracial 
state $\tau$ on $M$.   In this case it is a theorem that there exists a weak* continuous unital contractive projection $E$
from $M$ onto any  von Neumann subalgebra $D$ of $M$, and moreover there is a  unique such $E$ that is 
trace preserving (i.e.\ $\tau \circ E = \tau$).   We may write this $E$ as $E_\tau$, and call this
the `probabilistic' (we should 
perhaps say `tracial' here) conditional expectation.    The other weak* continuous unital contractive projections $E$
from $M$ onto  $D$ are again the `weightings' $E_\tau(hx)$ for densities $h \in L^1(M)_+$ which commute with 
$D$ and which satisfy $E_\tau(h) = 1$.  
The reader should note the parallel with $E_{\B}$ above,
the probabilistic conditional expectation.   Again we see that all weak* continuous unital contractive projections
from $M$ onto a  von Neumann subalgebra `are' weighted `probabilistic'  conditional expectations.
They may be viewed as a `partial integral with respect to a noncommutative measure'.

A brief noncommutative history of conditional expectations: von Neumann, Dixmier, Nakamura and
Turumaru, Umegaki, and others considered  conditional expectations in the framework of von Neumann (or
$C^*$-) algebras and established many properties of these objects (in particular the `Kolmogorov list' above, especially in
the context of von Neumann algebras with a finite trace (see e.g.\ \cite{AC} for references).  Some of these works were aiming to
generalize the Moy-Doob characterization mentioned above.   Tomiyama added the modern perspective of 
 conditional expectations in terms of norm one projections in $C^*$-algebras,  his theorem stated in our introduction 
 shows that all positive idempotents onto a $C^*$-algebra have the module property 
 that leads to them being called conditional expectations.   They are also completely positive and completely contractive as we said.
 See p.\ 132--133 in \cite{Blbook} for more on this and some 
other basic  facts about conditional expectations. 
 Others have 
 generalized some of the  work  of Moy-Douglas-Ando-Lacey and Bernau, etc., that we described above, to 
positive contractive projections on noncommutative $L^1$ or $L^\infty$ (i.e.\ on a von Neumann algebra $M$).
Conditional expectations play a profound role in the classification
 of von Neumann algebras,     e.g.\ in the structure theory of factors, or the fundamental work of Connes in which approximately finite von Neumann algebras are the 
amenable ones, and are the 
ones that are the range of a (not necessarily weak* continuous) conditional expectation  on $B(H)$.   Haagerup 
transferred conditional expectations to the powerful framework of operator valued weights and the extended positive part of a von Neumann algebra.
The latter is the noncommutative version of $(L_0)_+$, the positive measurable functions, and consists of suprema of increasing sequences of 
elements of $M_+$. 
This gives the most general perspective, allows treatment of general noncommutative $L^p$ spaces,
etc.   Conditional expectations are now a major and ubiquitous tool in the theory of
$C^*$- and von Neumann algebras, and there are by now a huge number of important examples 
(see e.g.\ \cite{Tak2,KadCE}  for more references).  

Nonetheless, outside of the class of von Neumann algebras with a faithful normal tracial 
state,  the existence of a weak* continuous conditional expectation onto a von Neumann subalgebra 
$D$ is a difficult question (unless $D$ is atomic).    Indeed this question gets to the heart of, and uses the whole industry of the theory of 
noncommutative integration (due to Connes, Haagerup, Pedersen, Takesaki, and very many other brilliant operator algebraists).   
See \cite{Tak2,LabG} for a taste of the latter.  For a commutative von Neumann algebra $L^\infty(K,\A,\mu)$ again, but with $\mu$ not $\sigma$-finite
 one must use the theory of semifinite measures to construct a conditional expectation onto a von Neumann subalgebra.
Now suppose that $M$ is a  noncommutative {\em semifinite von Neumann algebra}, for example $B(l^2)$.   Then $M$ has a faithful normal semifinite trace $\tau$.
However there need not exist any conditional expectation onto a fixed von Neumann subalgebra 
(e.g.\ it is known that there is no conditional expectation from  $B(l^2)$  onto nonatomic von Neumann subalgebras).
Indeed there exist a  $\tau$-preserving conditional expectation onto a  von Neumann subalgebra $D$ if and only if 
$\tau$ restricts to a semifinite trace on $D$.  The one direction of this is \cite[Proposition V.2.36]{Tak}.
For the other, if $0 \neq x \in D_+$ and $0 \neq y \in M_+$ with $\tau(y) < \infty$ and $y \leq x$,
then  $\tau(E(y)) = \tau(y) \in (0,\infty)$ and $E(y) \leq x$.  Clearly $0 \neq E(y) \in D_+$.    We have verified that 
the restriction of $\tau$ to $D$ is semifinite. 

For non-semifinite von Neumann algebras the situation is much more complicated, and gets into Haagerup's theory of 
operator valued weights (see \cite{Tak2,LabG} for references).    The conditions for existence of a weak* continuous conditional expectation onto a von Neumann subalgebra
are much more intricate, such conditions involving the operator semigroup central to Tomita-Takesaki modular theory.
See e.g.\ \cite[Theorem 4.2]{Tak2}.     For technical reasons and to avoid pathologies
one usually insists that  $M$ possesses a faithful normal state $\nu$
(which is equivalent to $M$ possessing a faithful state).   This class of von Neumann algebras 
includes those on a separable Hilbert space, or with separable predual.
 Then there exists a  $\nu$-preserving conditional expectation $E$ onto a  von Neumann subalgebra
 $D$ if and only if $D$ is invariant under the modular automorphism group $(\sigma_t^\nu)$ of $\nu$ (see \cite[Theorem IX.4.2]{Tak}).  
 Such  a  $\nu$-preserving conditional expectation is again unique.     We may write this $E$ as $E_\nu$, and again call this
the `probabilistic' conditional expectation (it depends on the fixed state $\nu$, which can be thought of as a noncommutative 
probability integral).   
We will not go into further detail here.

In summary we have seen that under certain conditions on a von Neumann algebra $M$ and a von Neumann subalgebra 
$D$, and on a positive functional or weight $\nu$ on $M$, conditions usually involving 
modular theory,  there exists a unique $\nu$-preserving 
weak* continuous unital contractive projection
from $M$ onto $D$.   We call this a conditional expectation, and it is analogous to $E_{\B}$ above.

\section{Noncommutative characters on  noncommutative function algebras (operator algebras)}  \label{BLvv} 

The last part  of our Positivity X lecture was concerned with ongoing joint work with L. E. Labuschagne \cite{BLv,BLv2}, 
 on a special case 
of the real positive projections considered  in the last Section   \ref{cp}.
This case we consider to be a good noncommutative
generalization of the classical theory of `characters' (i.e.\ homomorphisms into 
the scalars) of a function algebra (see e.g.\ \cite{Gam}).  Recall that if $A \subset C(K)$ is a function algebra or uniform algebra on compact set $K$,
then the fundamental associated object is the set $M_A$ of (scalar valued) characters on $A$.

A noncommutative function algebra for us is just an  operator algebra in the earlier sense,
a subalgebra $A$  of a $C^*$-algebra $C$.  We assume $C$ unital and $1_C \in A$ for simplicity here.
In the nonunital case we can unitize by the tricks in the early parts of Sections 2 and 3 above.
In this setting  scalar valued characters are usually not so useful, however 
we have found that in the following setting one can generalize many of the classical function algebra
character results.   Namely, consider an  inclusion $D \subset A \subset C$, where $A, C$ are as before, 
and $D$ is a  $C^*$-subalgebra of $A$. 
A   $D$-{\em character} is  a unital contractive homomorphism $\Phi : A \to D$ which is also a $D$-bimodule map (or equivalently,
is the identity map on $D$).
The classical scalar valued characters $\chi$ on $A$ fit into this setting by identifying $\chi$ with $\chi(\cdot) 1_A$.
We were motivated to study these because of their importance in (the definition of) 
{\em Arveson's subdiagonal algebras}   \cite{Arv2,BLsurv}.  Arveson also gives very many good examples of such $D$-characters in that paper.

Note that these fall within the framework of Section \ref{cp}, they are 
in fact automatically real completely positive completely contractive 
projections from $A$ onto a subalgebra.   That they are completely contractive  follows from the 
standard trick mentioned at the end of the proof 
of Theorem \ref{ipscor2}.   That they are real completely positive follows from e.g.\ Proposition \ref{onladdn}. 
Recall also that by e.g.\ Theorem \ref{crpproj}. a
completely contractive unital projection
onto a unital $C^*$-subalgebra $D$ is automatically a $D$-bimodule map.  

In this section we will for simplicity stick to the case of contractive unital characters.   In the nonunital case 
one would consider (completely) contractive real positive homomorphisms from 
the operator algebra $A$ onto a $C^*$-subalgebra.
As in Section \ref{cp} these extend uniquely to (completely) contractive unital homomorphisms on $A^1$, and so we are back in 
the unital character case.     Thus we may suppress discussion of real positivity in the next paragraphs: it is there but automatic.

In \cite{BLv,BLv2}, Labuschagne and the author consider  several problems that arise when generalizing classical function algebra results 
involving characters.     For the sake of the present article not becoming too scattered in theme we just mention briefly a couple of 
examples of these that use specific theorems from our earlier sections above.   
The first is a new noncommutative  take on the classical theory of Gleason parts of function algebras.
The Gleason relation ($\| \varphi - \psi \| < 2$) on characters of function algebras does not seem to have a $B(H)$ valued analogue suitable for our 
  purposes, but we show that interestingly  it does have a noncommutative variant for our $D$-characters. 
    
  We also use some concepts considered by Harris in e.g.\ \cite{Harris0,Harris}: 
$$T_{x}(y) = (1-xx^*)^{-\frac{1}{2}} (x+y) (1+x^*y)^{-1} 
(1-x^* x)^{\frac{1}{2}}.$$    
This makes sense for elements in the  open unit ball in $B(H)$.  For  fixed such $x$  the maps $T_x$ are essentially exactly the 
biholomorphic self maps of the  open unit ball in $B(H)$, or are 
{\em M\"obius maps} of
this open ball. 
The {\em hyperbolic distance} $\rho(x,y)$  is $$\tanh^{-1} \| (1-xx^*)^{-\frac{1}{2}} (x-y) (1-x^*y)^{-1} 
(1-x^* x)^{\frac{1}{2}} \| =  \tanh^{-1} \| T_{-x}(y) \|.$$
Harris shows   \cite{HarHol} that 
$\rho$ is what is known as a {\em CRF pseudometric}
on the open unit ball ${\mathcal U}_0$ 
and it satisfies the  {\em  Schwarz-Pick inequality} 
$$\rho(h(x),h(y)) \leq \rho(x,y), \qquad x, y \in {\mathcal U}_0,$$
for any holomorphic $h : {\mathcal U}_0 \to {\mathcal U}_0$.   We have equality here  if $h$ is biholomorphic. 

We may also define an equivalence relation using  the {\em real positive ordering}.   If $\Phi, \Psi$ are maps from $A$ into a $C^*$-algebra
  we write $\Phi \preceq \Psi$ if $\Psi - \Phi$ is a real positive map (in the sense of e.g.\ Section 3).    Note that 
  one may then show that in this situation it is real completely positive, and then apply Theorem \ref{stine2}.   This 
  permits us to define an equivalence relation on $D$-characters by the existence of strictly positive constants
$c, d$ with $\Phi \preccurlyeq 
c \Psi$ and
$\Psi \preccurlyeq d \Phi$.   The reasoning in the last few lines ties this equivalence relation with the famous notion of Harnack 
equivalence (see e.g.\ \cite{SV}, and we thank Sanne ter Horst for this and many other references).

Using these ideas, and following classical methods, and results like Theorem \ref{stine2} 
above, one may prove:

\begin{theorem}  Consider inclusions $D \subset A \subset C$ as above.  Suppose that $D$ is represented nondegenerately on a Hilbert space $H$.
 Let   $\Phi , \Psi : A \to D$ be  $D$-characters.   
 The following 
are equivalent: \begin{enumerate} \item  [{\rm (1)}]
$\| \Phi - \Psi \| < 2$.
\item  [{\rm (2)}] $\| \Phi_{|{\rm Ker} \, \Psi} \| < 1$.
\item  [{\rm (3)}]  There is a constant $M > 0$ with $\rho(\Phi(a), \Psi(a)) \leq M$ for $\| a \| < 1, a \in A$.
\item   [{\rm (4)}] If $\| \Phi(a_n) \| \to 1$ for a sequence $(a_n)$ in ${\rm Ball}(A)$, then 
 $\| \Psi(a_n) \| \to 1$.  \end{enumerate} 
The above conditions are implied by the equivalent conditions:
 \begin{enumerate} 
 \item [{\rm (5)}]  There are positive constants
$c, d$ with $\Phi \preccurlyeq
c \Psi$ and
$\Psi \preccurlyeq d \Phi$.  
\item [{\rm (6)}]  There are positive constants
$c, d$ and completely positive $B(H)$-valued
maps  $\tilde{\Phi}, \tilde{\Psi}$ extending
$\Phi, \Psi$ to $C$, with $\tilde{\Phi} \leq
c \tilde{\Psi}$ and
$\tilde{\Psi} \leq d \tilde{\Phi}$.  
\end{enumerate}
If $D$ is 1 dimensional then all  the conditions here are equivalent.   
\end{theorem}

At the time of writing we do not know if the conditions in the last theorem are equivalent in full generality.
That is, we do not know if we have two distinct equivalence relations in the general case.
Gleason parts are applied in \cite{BLv}  to the theory of Hankel and Fredholm Toeplitz operators. 

\begin{remark}   If $\Phi : A \to A$ is a completely contractive unital projection then $\Phi(A)$ is an operator algebra
in the $\Phi$ product and $\Phi$ is a `$\Phi(A)$-bimodule map' with respect to that operator algebra, by Theorem \ref{crpproj}.  
If $\Phi(A) \subset \Delta(A)$ then by the idea in the proof of Corollary \ref{ipscor} but appealing to the Choi-Effros in our 
introduction instead of to Effros-St{\o}rmer, one sees that $\Phi(A)$ is selfadjoint, a $C^*$-algebra in the $\Phi$ product. 
This is not mentioned in \cite{BLv} but many of  the results in that paper, including parts of the theorem above, will go through for such maps. 
\end{remark} 

Finally we discuss {\em noncommutative representing measures} for $\Phi : A \to D$, a $D$-character.
  A positive measure $\mu$  on  a set $K$ is called a \emph{representing measure} for a character $\Phi$  of a function algebra 
  $A$ on $K$ if
$\Phi(f)=\int_K \, f \,d\mu$ for all $f\in A$. The functional $\widetilde{\Phi}(g) = \int_K \, g \, d\mu$ on $C(K)$ is a state on $C(K)$,
and indeed representing measures for $\Phi$ are in a bijective correspondence with the extensions  of $\Phi$ 
to a positive functional on $C(K)$.  That is, representing measures for  a character are just the Hahn-Banach extensions to $C(K)$ of that character.

Noncommutative representing measures will therefore be related somewhat to the earlier theorem  \ref{stine2} concerning positive extensions.
Suppose that we are given a faithful representation of $D$ on a Hilbert space $H$.  
The usual noncommutative analogue of a `noncommutative representing measure' for say
a $D$-character on $A$ would be a $B(H)$-valued 
extension of $\Phi$ to a $C^*$-algebra $B$ containing $A$, which is 
completely positive (or equivalently, in this case, completely contractive).    Such noncommutative representing measures $\Psi : B \to B(H)$ always exist,
by Theorem  \ref{stine2} (indeed by Arveson's extension theorem \cite[Theorem 1.2.9]{Arv}).  However although these 
noncommutative notions are appropriate in  many settings, 
they do not necessarily seem appropriate when generalizing some other important parts of the theory
of uniform algebras.   An intuitive reason we advance for now for this (other reasons will become clearer momentarily) is that 
$B(H)$ is too big, thus insensitive; in some settings one probably would not want to go too far from $D$ in the range 
if one does not have to.   We shall see below that for some purposes  one should not have to. 
 
 The alternative {\em noncommutative representing measure} that we are proposing, again inspired by
Arveson (but this time his noncommutative analyticity work \cite{Arv2}),  is
a completely positive extension to $B$ that takes values in $D$ (or possibly a weak* closure of $D$).   
Let us call these {\em tight   noncommutative representing measures}.  Now however one has to face 
the problem of existence of such an extension.  Such existence  would in a real sense improve on  Theorem  \ref{stine2}
in the case of $D$-characters.  This problem is dealt with by exploiting the $C^*$-algebraic or von Neumann algebraic theory of 
conditional expectations from $B$ onto $D$. 

In the following discussion,   we have weak*-continuous unital inclusions $D \subset A \subset M$, where 
$M$ is a von Neumann algebra, and $A$ and $D$ are unital weak* closed subalgebras, with $D$ selfadjoint
(hence a von Neumann subalgebra).   
We are also given a weak*-continuous $D$-character $\Phi : A \to D$.    We seek a weak* continuous positive extension $\Psi : M \to D$.    
Now we can see that we are asking for something
quite interesting in several ways.   Firstly, we are asking  for a generalization of the 
remarkable and deep theory of von Neumann algebra conditional expectations summarized briefly in Section \ref{wace}.
Indeed setting $A = D$, the question above becomes precisely the important question of the 
existence of a normal (i.e.\ weak* continuous) expectation of a fixed von Neumann algebra onto a von Neumann subalgebra.  
Second, it is interesting because weak* continuous positive extensions of weak* continuous  {\em linear} unital contractive maps do not typically exist.
Indeed saying `positive' here is equivalent to saying `contractive', and even in the case that the range is one dimensional
(i.e.\ $D = \Cdb 1$) the Hahn-Banach theorem about extensions with the same norm 
usually fails drastically if all maps are supposed to be weak* continuous.   This point is discussed early in \cite{BFZ}, and we will end our paper
with an example of such failure.   It is important that $\Phi$ is a homomorphism for such a positive weak* continuous extension to exist. 
Third, this is precisely the setting of Arveson's famous paper \cite{Arv2}  on noncommutative generalizations 
of Hardy spaces.   In  Arveson's approach to noncommutative analyticity/generalized analytic functions/Hardy spaces
we have a normal  conditional expectation $\Psi : M \to D$ extending a $D$-character $\Phi$ on $A$.   In this
`generalized analytic function theory' it is very important that the representing measures are $D$-valued rather than
$B(H)$-valued.

In the classical case if $\mu$ is a representing (probability) 
measure on a space $K$ for a character $\theta$ of a function algebra $A$ on $K$, we define $H^\infty(\mu)$ to 
be the weak* closure of $A$ in $L^\infty(\mu)$.  Similarly for $p< \infty$ define
$H^p(\mu)$ to be the  closure of $A$ in $L^p(\mu)$.  E.g.\  if $A$ is the disk algebra or $H^\infty$ of the disk, then 
$\theta(f) = f(0)$, and the important `representing measure' is $\mu(f) = \int_{\Tdb} \, f \, dm$, Lebesgue integration on the circle, which is 
a state on $C(\Tdb)$ and a weak* continuous state on $L^\infty(\Tdb)$.  If $A$ is a  
{\em Dirichlet} or {\em logmodular algebra}, and indeed much more generally,  these Hardy spaces behave very similarly to 
the classical Hardy spaces  of the disk.   One obtains an  F \& M Riesz theorem, 
Beurling's theorem, Jensen and Szego theorems, Gleason-Whitney theorem, inner-outer factorization, and so on.
Arveson was attempting a vast noncommutative generalization of all of this, using
 precisely the noncommutative representing measure of a $D$-character approach that we are describing.
 
Arveson gave many interesting examples, showing that his framework synthesized several theories that were emerging in the
1960's.  Work on Arveson's spaces has continued over the decades by very many authors
(see e.g.\ \cite{BLsurv} for many references), being at  present something of an international industry. His vision was
 realized in the case that $A + A^*$ is weak* dense in $M$ (again, see e.g.\ \cite{BLsurv}).   We are 
trying to push this same noncommutative representing measure approach to operator algebras beyond the latter case. 

Returning to tight noncommutative representing measures for a weak* continuous
$D$-character $\Phi : A \to D$, the primary problem concerns their existence, which turns out to hold for quite subtle reasons.
We seek a weak* continuous positive extension $\Psi : M \to D$ of $\Phi$, where 
$M$ is a fixed von Neumann algebra containing $A$ unitally, and contains $D$ as a von Neumann subalgebra.  
In \cite{BFZ} this is done if $D$ is atomic (and it is explained there why this is a noncommutative generalization of an old
 theorem of Hoffman and Rossi).  

\begin{theorem}  \label{BLvvf} {\rm \cite{BLv2}}  Consider weak*-continuous unital inclusions $D \subset A \subset M$, where 
$M$ is a von Neumann algebra which is  commutative, or which possesses
a faithful normal tracial state, and $A$ and $D$ are unital weak* closed subalgebras, with $D$ selfadjoint.   
  If  $\Phi : A \to D$ is a weak*-continuous $D$-character then $\Phi$ has a 
weak* continuous positive extension $\Psi : M \to D$.
\end{theorem}

The main point is that this suggests that at least for some purposes one should not need, and probably should not use, general $B(H)$-valued extensions
of $D$-characters.   One in fact has the (surprising) existence of tight noncommutative representing measures.
Indeed in the above theorem we have this existence for any von Neumann subalgebra $D$.  

We also have a much more general theorem giving
existence of  weak*-continuous representing measures \cite{BLv2}.    It is similar to the last result: there exist
  weak* continuous positive extensions $\Psi : M \to D$ 
of weak*-continuous $D$-characters on weak* closed unital subalgebras of $M$
 for much more general classes of von Neumann algebras $M$ (without a faithful normal tracial state).
These results require extra conditions
e.g.\ on modular automorphism groups in the same spirit as the second last paragraph of Section {\rm \ref{wace}}.

That is, there exists a much more general, but considerably more technical, version 
of Theorem \ref{BLvvf}.   Those familiar with the conditions from Tomita-Takesaki theory ensuring the existence of  von Neumann algebraic 
conditional expectations (see e.g.\ \cite[Theorem 4.2]{Tak2})  will be able to guess what the vague conditions are.     Those not versed in modular theory would not 
be thankful 
for an  explicit statement of these conditions!   
Basically we are saying that `noncommutative representing measures' exist in this setting, under basically the same conditions that von Neumann algebra conditional expectations exist.   The proof uses, in additional to the 
arsenal of noncommutative integration theory alluded to above, techniques from \cite{Lab, LabG,Xu} and elsewhere, as well as new ideas.

\begin{example} 
We end our paper with an example showing the necessity of using 
$D$-characters $\Phi : A \to D$ in the results above, even in the scalar valued case
($D = \Cdb$).
The suspicious reader might think that possibly the issue is that $A$ 
needs to be an algebra, rather than that $\Phi$ needs to be a homomorphism.
Take any subspace ${\mathcal S}$
of $L^\infty([0,1]$ possessing a 
norm $1$ functional $\varphi_1$ with no weak* continuous
 Hahn Banach (state)
extension to $N = L^\infty([0,1])$.  For example, 
the polynomials of degree $\leq 1$ with $\varphi_1(p) = p(1)$ will do.   
 In \cite[Proposition 2.6]{BFZ} we considered the 4 dimensional subalgebra
$A$ of $M_2(L^\infty([0,1])$ consisting of upper triangular $2 \times 2$ matrices
with constant functions on the diagonal, and an element from  ${\mathcal S}$ 
in the $1$-$2$ position.    We will construct a
weak* continuous state (hence a completely contractive $D$-module map
onto $D$) on $A$  with
 no weak* continuous Hahn-Banach extension to $M_2(N)$.

To do this let $s \in (0,1)$.  We will use the fact that the functional on the upper triangular $2 \times 2$ matrices which takes $E_{11} \mapsto s, E_{22} \mapsto 1-s, E_{12} \mapsto \mu$, for $s \in [0,1]$, is contractive (and hence is a state) if and only if $|\mu|^2 \leq s(1-s)$.  To see this note it is contractive
if and only if it is a state.   States on
the upper triangular $2 \times 2$ matrices are easily seen to have
 unique state extensions
to $M_2$.   Indeed there is a bijectiion between these two state spaces.
States on $M_2$ correspond to density matrices, that is positive matrices of trace $1$.
These are selfadjoint matrices with $s, 1-s$ on the diagonal for $s \in [0,1]$, and off diagonal entries coming
from a scalar $\mu$ with $|\mu|^2 \leq s(1-s)$.

We may scale: let  $\psi = \sqrt{s(1-s)} \, \varphi_1$.   Then the functional
$\varphi$ on $A$ defined by  
$$\varphi \bigl( \left[ \begin{array}{ccl} a & x \\ 0 & c \end{array} \right] 
\bigr) =  \;
sa + (1-s) c + \psi(x) , \quad a, c \in \Cdb, x \in X ,$$
is a weak* continuous  state on $A$.  This uses the fact
(clear from the formula (2.1) in 
\cite{BM05}) that if $z$ is the first matrix in the last displayed equation, then $z$
has the same norm as the same matrix but with $x$ replaced by $\| x \|$, and $a$ and $c$ replaced by their
modulus.   Since $\| \psi \| = \sqrt{s(1-s)}$,  we have by the
last paragraph that
$$s|a| + (1-s)|c| + | \psi(x)|
= s|a| + (1-s)|c| + \| \psi \| \| x \|  \leq
\| \left[ \begin{array}{ccl} |a| & \| x \| \\ 0 & |c| \end{array} \right]  \| = \| z \|.$$
Thus $|\varphi(z) | \leq \| z \|$.
We claim that $\varphi$ has no weak* continuous Hahn-Banach extension to $M_2(N)$.
Indeed if there were, then we obtain a  weak* continuous Hahn-Banach extension to 
$B$, the set of upper triangular $2 \times 2$ matrices with  constant functions on the diagonal, and an element from
 $N$ in the $1$-$2$ position.
Thus we obtain a  weak* continuous extension $\xi$ of $\psi$ to 
$N$, such that  the functional
$$\left[ \begin{array}{ccl} a & x \\ 0 & c \end{array} \right] \; \mapsto \;
sa + (1-s) c + \xi(x) , \quad a, c \in \Cdb, x \in N ,$$
is a weak* continuous  state on $B$. If $x \in N$ with $\| x \| = 1$ then for any scalars $a, b, c$ we have
$$|sa + (1-s)c + b \xi(x)| \leq \| \left[ \begin{array}{ccl} a & b  x  \\ 0 & c \end{array} \right]  \|
= \| \left[ \begin{array}{ccl} a & \| b  x \| \\ 0 & c \end{array} \right]  \| =  \| \left[ \begin{array}{ccl} a & b  \\ 0 & c \end{array} \right]  \|.$$   We have used again 
formula (2.1) in
\cite{BM05}, twice. 
Thus by the last paragraph $|\xi(x)| \leq \sqrt{s(1-s)}$, so that
$\| \xi \| \leq \sqrt{s(1-s)}$.
  This contradicts the fact that $\psi$ has no weak* continuous
Hahn Banach extension to $N$.
\end{example}

\medskip

\subsection*{Acknowledgment}
 We thank Louis Labuschagne and Matt Neal for very many discussions.   We also thank Angel 
Rodr\'iguez Palacios for many very useful comments and references.   Some of our results on contractive projections, and some complementary results,
 he has subsequently 
extended in unpublished work \cite{RodAz} to the class called  {\em Arazy algebras} and introduced in \cite{CR18}; in particular to
unital complete normed power-associative complex algebra satisfying von Neumann's inequality (we will not define all these terms). 
Some related theory of nonunital  nonassociative algebras may be found in the last section of \cite{CRM}.
 While writing this article we learned of the death of Ed Effros, many of whose 
beautiful and important ideas are featured here, and were struck again by his profound contributions to the subject.   Similarly, 
we often fondly remembered Coenraad Labuschagne during this writing
  for his warm and kind personality, and for his fine work on the subject of conditional expectations.


\begin{thebibliography}{99}

 \bibitem{AC} L. Accardi and C. Cecchini, {\em Conditional expectations in von Neumann algebras and a theorem of Takesaki,} J. Functional Analysis {\bf 45} (1982),  245--273.

\bibitem{Av}   J. Arazy, {\em Isometries of Banach algebras satisfying the von Neumann inequality,} Math. Scand. {\bf 74} (1994), 137--151. 


\bibitem{AS}   J. Arazy and B. Solel, {\em
Isometries of nonselfadjoint operator algebras,}
J. Funct. Anal. {\bf 90} (1990),  284--305.



\bibitem{Arv2}   W. B. Arveson, {\em Analyticity
in operator algebras,} Amer.\ J.\ Math.\  {\bf 89} (1967),
578--642.



\bibitem{Arv}   W. B. Arveson,
{\em Subalgebras of $C^{*}-$algebras,}  Acta Math. {\bf 123 }
(1969), 141--224.


\bibitem{BBS}  C. A. Bearden, D. P. Blecher and S. Sharma, {\em On positivity and roots in operator algebras,}  Integral Equations Operator Theory {\bf  79} (2014),  555--566.


\bibitem{Blbook} B. Blackadar,  Operator algebras.
Theory of $C\sp *$-algebras and von Neumann algebras,
Encyclopaedia of Mathematical Sciences, 122, Springer-Verlag, Berlin, 2006.


\bibitem{Bsan} D. P. Blecher,  {\em Generalization of C*-algebra methods via real positivity for operator  and Banach algebras,}
 pages 35--66
in "Operator algebras and their applications: A tribute to Richard V. Kadison", (ed.\ by R.S.\ Doran and E.\ Park),
vol. 671, Contemporary Mathematics, American Mathematical Society, Providence, R.I.\
2016.


\bibitem{BFZ} D. P. Blecher, L. C. Flores, and B. G. Zimmer, {\em The Hoffman-Rossi theorem for operator algebras,}   Integral Equations Operator Theory {\bf 91} (2019), 
no. 2, Art. 17, 7 pp.   


\bibitem{BLsurv} D. P. Blecher
and L. E. Labuschagne, {\em 
Von Neumann algebraic $H^p$ theory,} in 
{\em Function Spaces: Fifth Conference on Function Spaces,}
Contemp. Math. Vol.\ 435, Amer.\ Math.\ Soc.\ (2007), 89-114.


\bibitem{BLv} D. P. Blecher
and L. E. Labuschagne, {\em On vector-valued characters for noncommutative function algebras}, Complex Analysis and Operator Theory {\bf 14} (2020) https://doi.org/10.1007/s11785-020-00989-1.


\bibitem{BLv2} D. P. Blecher
and L. E. Labuschagne, {\em On the existence of  noncommutative representing measures for dual operator algebras} (tentative title), in preparation, 2019.

\bibitem{BLM}  D. P. Blecher
and C.  Le Merdy, {\em Operator algebras and their modules---an
operator space approach,} Oxford Univ.\  Press, Oxford (2004).


\bibitem{BM05a}  D. P. Blecher
and B. Magajna, {\em Duality and operator algebras: 
automatic weak* continuity and applications}, J. Funct.\ Analysis {\bf 224}
 (2005), 386--407.

\bibitem{BM05}  D. P. Blecher
and B. Magajna, {\em  Duality and operator algebras II: Operator algebras as Banach algebras}, J. Funct.\ Analysis 
{\bf 226} (2005), 485--493.


\bibitem{BNp}   D. P. Blecher
and M. Neal, {\em Completely contractive projections on operator algebras,}  Pacific J. Math.   {\bf 283-2} (2016), 289--324.

\bibitem{BNj}   D. P. Blecher
and M. Neal, {\em  Noncommutative topology and Jordan operator algebras,}  Mathematische Nachrichten {\bf
292} (2019), 481--510.

\bibitem{BNjp}   D. P. Blecher
and M. Neal, {\em Contractive projections and real positive maps on  operator algebras},
 (2019), to appear Studia Math. 

\bibitem{BOZ}  D. P. Blecher
and  N. Ozawa, {\em Real positivity and approximate identities in Banach algebras,} Pacific Math. J. {\bf 277} (2015), 1--59. 

\bibitem{BPhilI}   D. P. Blecher and N. C. Phillips, {\em $L^p$ operator algebras with approximate identities I,}   Pacific Journal of Mathematics. 
{\bf 303} (2019) dx.doi.org/10.2140/pjm.2019.303.401

\bibitem{BPhilII}   D. P. Blecher and N. C. Phillips, {\em $L^p$ operator algebras with approximate identities II,}  In preparation.  


 \bibitem{BRI}  D. P. Blecher and C. J. Read, {\em  Operator algebras with contractive approximate identities,}
J. Functional Analysis {\bf 261} (2011), 188--217.

\bibitem{BRII}  D. P. Blecher and C. J. Read, {\em  Operator algebras with contractive approximate identities II,} J. Functional Analysis {\bf 264} (2013), 1049--1067.

\bibitem{BRord}  D. P. Blecher and C. J. Read, {\em   Order theory  and interpolation in operator algebras,}   Studia Math. {\bf 225} (2014), 61--95.


\bibitem{BRS}  D. P.  Blecher, Z-J. Ruan, and A. M.
Sinclair, A characterization of operator algebras, {\em  J.\ Funct.\
Anal.\ } {\bf 89} (1990), 188--201.

\bibitem{BT}  D. P.  Blecher and W. Tepsan, {\em Real operator algebras and real positive maps,} Preprint July 2020, arXiv:2007.12259

\bibitem{BWj}  D. P. Blecher and Z. Wang, {\em Jordan operator algebras: basic theory,} Mathematische Nachrichten, {\bf 291} (2018), 1629--1654.

\bibitem{BWj2}  D. P. Blecher and Z. Wang, {\em Jordan operator algebras revisited,}
 Mathematische Nachrichten {\bf 292} (2019), 2129--2136.

\bibitem{DR} N. Dinculeanu and M. M. Rao, {\em
Contractive projections and conditional expectations,} 
J. Multivariate Anal. {\bf 2} (1972), 362--381. 

\bibitem{Gam}   T. W. Gamelin,
 {\em Uniform Algebras,} Second edition, Chelsea, New
York, 1984.

\bibitem{Rod}  M. Cabrera and A.  Rodr\'iguez, {\em
        Non-associative normed algebras. Vol. 1,}
The Vidav-Palmer and Gelfand-Naimark theorems. Encyclopedia of Mathematics and its Applications, 154. Cambridge University Press, Cambridge, 2014.

\bibitem{Rod2}  M. Cabrera and A.  Rodr\'iguez, {\em
        Non-associative normed algebras. Vol. 2, Representation Theory and the Zel'manov Approach}. Encyclopedia of Mathematics and its Applications, 167. Cambridge University Press, Cambridge, 2018.


\bibitem{CR18} M. Cabrera and A. Rodr\'iguez, {\em The von Neumann inequality in complete normed non-associative complex algebras,} Math. Proc. R. Ir. Acad. {\bf 118A} (2018), 83--125.

\bibitem{CRM} M. Cabrera and A. Rodr\'iguez, {\em Multiplication Algebras: Algebraic and  Analytic Aspects,} In "Associative and Non-Associative Algebras and  Applications, 3rd Moroccan Andalusian Meeting on Algebras and their  Application" (Ed. M. Siles, E. K. Laiachi, M. Louzari, L. M. Ben Yakoub,  and M. Benslimane) 113-138, Springer Proceedings in Mathematics and  Statistics Vol. 311, Springer, 2020.


\bibitem{CR20} M. Cabrera and A. Rodr\'iguez, {\em Unit-free contractive projection
theorems for C*-, JB*-, and
JB-algebras,} J. Math. Anal. Appl.  {\bf 486} (2020),  123--921.


\bibitem{CE}  M.-D. Choi and E. G. Effros, {\em Injectivity and
operator spaces,} J.\ Funct.\ Anal.\  {\bf 24}
(1977), 156--209.

\bibitem{Doug} R. G. Douglas, {\em Contractive projections on an $L^1$ space,} Pacific J. Math.  {\bf 15} (1965), 443--462. 

\bibitem{ES}
E. G. Effros and E. St{\o}rmer, \emph{Positive projections and Jordan structure in operator algebras,} Math. Scand. 45 (1979),  127--138.




\bibitem{FRAD}  Y. Friedman and B.  Russo, {\em
Conditional expectation without order,}  Pacific J. Math.
{\bf  115} (1984), 351--360.

\bibitem{CPP}  Y. Friedman and B.  Russo,  {\em  Solution of
the contractive projection problem,} J.\ Funct.\ Anal.\  {\bf 60}
(1985), 56--79.

\bibitem{FRceJ}  Y. Friedman and B.  Russo,  {\em  Conditional expectation and bicontractive projections on Jordan $C^*$-algebras and their generalizations,} Math.\ Z.\ {\bf  194} (1987), 227--236.


\bibitem{HS}  H. Hanche-Olsen and E. St{\o}rmer, {\em
Jordan operator algebras,}
Monographs and Studies in Mathematics, 21. Pitman (Advanced Publishing Program), Boston, MA, 1984.


\bibitem{Harris}  L. A. Harris,  {\em Banach algebras with involution and M\"obius transformations,} J.\ Functional Analysis \textbf{11} (1972), 1-16.


\bibitem{Harris0}  L. A. Harris, {\em Bounded symmetric homogeneous
domains in infinite dimensional spaces,}   Lecture Notes in
Math., {\bf 364}, Springer-Verlag, Berlin - New York (1973).

\bibitem{HarHol}  L. A. Harris, {\em Schwarz-Pick systems of pseudometrics for domains in normed linear spaces,}  in {\em Advances in holomorphy (Proc. Sem. Univ. Fed. Rio de Janeiro, Rio de Janeiro, 1977),} pp. 345--406, North-Holland Math. Stud., 34, North-Holland, Amsterdam-New York, 1979.

\bibitem{Har2}  L. A. Harris, {\em  A generalization of $C^*$-algebras,} Proc. Lond. Math. Soc. {\bf 42}  (1981),  331--361.

\bibitem{HHH} M. Horodecki, P.  Horodecki, and R. Horodecki, {\em  Separability of mixed states: necessary and sufficient conditions,}  Phys. Lett. A {\bf 223} (1996), 1--8.

\bibitem{JVW}  P. Jordan, J. von Neumann and E. Wigner, {\em 
On an algebraic generalization of the quantum mechanical formalism,} 
Ann. of Math. {\bf 35} (1934),  29--64. 

\bibitem{K84} W. Kaup, {\em Contractive projections on Jordan C*-algebras and
generalizations,} Math. Scand. {\bf 54} (1984), 95--100.

 \bibitem{Kad}   R. V. Kadison, {\em  Isometries of
operator algebras,} Ann.\ of Math.\ {\bf 54} (1951), 325--338.

\bibitem{KadCE}   R. V. Kadison, {\em Non-commutative conditional expectations and their applications,} in Operator algebras, quantization, and noncommutative geometry, 143--179, Contemp. Math., 365, Amer. Math. Soc., Providence, RI, 2004. 

\bibitem{Lab} L. E. Labuschagne, {\em  Invariant subspaces for $H^2$ spaces of $\sigma$-finite algebras,} Bull. Lond. Math. Soc. {\bf
49} (2017),  33--44.

\bibitem{LabG} L. E. Labuschagne and S. Goldstein, {\em Notes on noncommutative $L^p$ and Orlicz spaces,} submitted ebook manuscript, 2020.

\bibitem{Lacey}   H. E. Lacey, {\em The isometric theory of classical
Banach spaces,}  Die Grundlehren der Math.\ Wissenschaften,
208, Springer-Verlag Verlag New York-Heidelberg, 1974.



\bibitem{LL} A. T. Lau and  R. J. Loy, {\em Contractive projections on Banach algebras,} J. Funct. Anal. {\bf 254} (2008), 2513--2533.


\bibitem{M}  J. D. Maitland Wright, {\em Jordan $C^*$-algebras,}   Michigan Math. J. {\bf 24} (1977), 291--302.

\bibitem{Moy} S.-T. C. Moy, {\em Characterizations of conditional expectation as a transformation on function spaces,}  Pacific J. Math. {\bf  4} (1954), 47--63. 

\bibitem{NgOz} P.  Wong and N. Ozawa, {\em 
A characterization of completely 1-complemented subspaces of noncommutative $L^1$-spaces,}
Pacific J. Math. {\bf 205} (2002), 171--195. 

\bibitem{Pau}   V. I. Paulsen, {\em Completely bounded maps and operator
algebras,} Cambridge Studies in Advanced Math., 78, Cambridge
University Press, Cambridge, 2002.

\bibitem{P} G. K. Pedersen, {\em $C^*$-algebras and their automorphism
groups,} Academic Press, London (1979).

\bibitem{R83} A. Rodr\'iguez, {\em Non-associative normed algebras spanned by hermitian
elements,} Proc. Lond. Math. Soc. {\bf 47} (1983), 258--274.

\bibitem{R98} A. Rodr\'iguez,  {\em Isometries and Jordan isomorphisms onto C*-algebras,}
J. Oper. Theory {\bf 40} (1998), 71--85.


\bibitem{Rod3} A. Rodr\'iguez, {\em Approximately norm-unital products on
$C^*$-algebras, and a non-associative Gelfand-Naimark theorem,}
 J. Algebra {\bf 347} (2011), 224--246.  
 
 \bibitem{RodAz} A. Rodr\'iguez, {\em Conditional expectations in complete normed complex algebras satisfying the von Neumann 
 inequality,} Unpublished work in progress, 2020.

\bibitem{Sti}  W. F. Stinespring, Positive functions on
$C^*$-algebras,  {\em Proc. Amer. Math. Soc.}  {\bf 6} (1955),
211--216.

\bibitem{PLMO}  E. St{\o}rmer,   {\em Positive linear maps of operator
algebras,} Acta Math.\  {\bf  110  } (1963), 233--278.

\bibitem{ST82}  E. St{\o}rmer, {\em Positive projections with contractive complements on $C^*$-algebras,} J. London Math. Soc. {\bf 26} (1982),  132--142.


\bibitem{ST}  E. St{\o}rmer, {\em Positive linear maps of operator algebras,} Springer Monographs in Mathematics, Springer-Verlag (2013).



\bibitem{SV} I. Suciu and I. Valusescu, {\em On the hyperbolic metric on Harnack parts,} Studia Math. {\bf 55} (1975/76),  97--109.


\bibitem{Tak} M. Takesaki, {\em Theory of Operator Algebras I},
Springer, New York, 1979.

\bibitem{Tak2} M. Takesaki, {\em Theory of Operator Algebras II},  Encyclopaedia of Mathematical Sciences, 125. Operator Algebras and Non-commutative Geometry, 6. Springer-Verlag, Berlin, 2003. 

\bibitem{Xu} Q. Xu, {\em On the maximality of subdiagonal algebras}, J. Operator Theory {\bf 54} (2005) 137--146.


\bibitem{ZWthes}  Z. Wang, {\em Theory of Jordan operator algebras and operator $*$-algebras,} PhD thesis, University of Houston, 2019.

 \end{thebibliography}
\end{document}